\documentclass[a4paper]{amsart}
\usepackage{mathtools,booktabs,amssymb,stmaryrd,latexsym,epic,bbm,mathrsfs}
\usepackage{comment}
\usepackage{indentfirst,color}
\usepackage[none]{hyphenat}\usepackage{xcolor}
\usepackage[numbers,sort&compress]{natbib}

\numberwithin{equation}{section}
\newtheorem{theorem}{Theorem}[section]
\newtheorem{lemma}[theorem]{Lemma}
\newtheorem{corollary}[theorem]{Corollary}
\newtheorem{claim}{Claim}

\newtheorem{definition}[theorem]{Definition}
\newtheorem{proposition}[theorem]{Proposition}

\newfont{\bbb}{cmb10 at 11pt}
\usepackage[parfill]{parskip}
\makeatletter
\def\subsection{\@startsection{subsection}{2}%
  \z@{1.5ex \@plus .5ex \@minus .2ex}%
  {1ex}%
  {\normalfont\bfseries}}
\makeatother

\title{Modules over the Takiff Block type Lie algebra}
\author[S.~Xu]{Shuoyang Xu}
\address{School of Mathematics, Hohai University, Nanjing 210098, China}
\email{shuoyang@hhu.edu.cn}

\author[H.~Chen]{Haibo Chen$^\ast$}
\address{School of Science, Jimei University, Xiamen, Fujian 361021, China}
\email{hypo1025@jmu.edu.cn}
\thanks{$^\ast$ Corresponding author: hypo1025@jmu.edu.cn}

\begin{document}
	\begin{abstract}
    In this paper, we introduce a family of infinite-dimensional Lie algebras $\mathcal B(q)$ of Takiff Block type, containing Heisenberg-Virasoro algebra, $W$-algebra $W(2,2)$, and BMS-Kac-Moody algebra as subalgebras. For $q\in\mathbb C^*$, we classify the $\mathcal B(q)$-modules that are free of rank one over $U(\mathfrak h)$, where $\mathfrak h=\mathbb C L_{0,0}\oplus\mathbb C W_{0,0}$, and determine their irreducibility and isomorphism classes. We also establish irreducibility and isomorphism criteria for their tensor products with irreducible restricted modules. Finally, restricting the above $\mathcal B(-1)$-modules to several natural subalgebras of the central quotient $\mathcal B(-1)/(\mathbb C C_1\oplus\mathbb C C_2)$ yields families of non-weight modules.
	\end{abstract}
	\maketitle

\tableofcontents

\section{Introduction}
Block introduced a class of infinite-dimensional simple Lie algebras in \cite{B}. Various generalizations, commonly called \textit{Lie algebras of Block type}, have since been studied; see, e.g., \cite{OZ,Z,SXX13,SXX12,CGZ14,CG13,XZ13,CY18}. Lie algebras of Block type are closely related to the Virasoro and Virasoro-like algebras, and some arise as special cases of generalized Cartan-type $S$ and $H$ Lie algebras~\cite{X}. To better understand the representation theory of generalized Cartan-type Lie algebras, it is natural to begin with such Block type examples, and in particular with their non-weight modules. Among these, \(U(\mathfrak h)\)-free modules constitute an important class.

These modules were introduced by Nilsson~\cite{Nil15} and were also studied from a different point of view in~\cite{TZ18}. Nilsson showed that a finite-dimensional simple Lie algebra admits nontrivial $U(\mathfrak h)$-free modules if and only if it is of type $A$ or $C$~\cite{Nil15,Nil16}. A family of rank one $U(\mathfrak h)$-free modules over the Virasoro algebra was constructed in~\cite{GLZ13}. Since then, such modules have been constructed for various Lie (super)algebras; see, e.g., \cite{TZ15,CG17,HCS17,CG26}. Notably, rank-one free modules over $W(2,2)$ and Lie algebras of Block type were classified in~\cite{CG17} and~\cite{CY18}, respectively.

The Takiff construction adjoins to a Lie algebra an abelian copy of itself~\cite{Tak71}. Modules over the Takiff algebra of $\mathfrak{sl}_2$ were studied in~\cite{Zhu24,Zhu25}, and tensor products of such modules were considered in~\cite{QZ26}. For the Virasoro algebra, the Takiff construction gives $W(2,2)$, while adjoining a nonnegative second index to its generators gives the Block type Lie algebra.
The algebra $\mathcal B(q)$ studied in the present paper combines these two constructions. It is the Takiff extension of the Block type Lie algebra and contains both $W(2,2)$ and the Block type Lie algebra as subalgebras. More precisely, its subalgebra spanned by the generators with second index zero is isomorphic to $W(2,2)$. Hence every rank-one free $\mathcal B(q)$-module restricts to a rank-one free $W(2,2)$-module. The classification problem is therefore to determine which rank-one free $W(2,2)$-modules extend to $\mathcal B(q)$ and how the generators with positive second index can act. The compatibility conditions depend on $q$ and may give rise to additional module structures.

In this paper, we classify all $\mathcal B(q)$-modules that are free of rank one over $U(\mathbb C L_{0,0}\oplus\mathbb C W_{0,0})$ and give explicit formulas for their actions. A family of particular interest arises when $q=1/n$ for some positive integer $n$. We also determine the irreducibility and isomorphism classes of these modules. Following the tensor product construction for Virasoro modules in~\cite{TZ13}, we further study their tensor products with irreducible restricted $\mathcal B(q)$-modules and establish corresponding irreducibility and isomorphism criteria.

As applications, we consider the restrictions of these modules to several natural subalgebras of \(\mathcal B(-1)/(\mathbb C C_1\oplus\mathbb C C_2)\), including the rank-one extended \(W(2,2)\) algebra, the Heisenberg-Virasoro algebra, and its Takiff Lie algebra, and thus obtain families of non-weight modules together with explicit irreducibility criteria. Finally, in light of the results of~\cite{CH20}, we construct admissible tensor product modules over the Heisenberg-Virasoro subalgebra.

Throughout this paper, we denote by $\mathbb{C},\mathbb{C}^{*},\mathbb{Z},\mathbb{Z}^*,\mathbb{N}$ and $\mathbb{Z}_+$  the sets of all complex numbers, nonzero complex numbers, integers, nonzero integers, non-negative integers, and positive integers respectively. All vector spaces, modules and Lie algebras are over $\mathbb{C}$.
For a Lie algebra $\mathcal{L}$, we use $U(\mathcal{L})$ to denote the universal enveloping algebra.

\section{The known results}
In this section, we recall a class of non-weight modules over the $W$-algebra $W(2,2)$ defined in \cite{CG17}.

The
$W$-algebra $W(2,2)$ was originally introduced in \cite{ZD}, where it served as a foundational tool for investigating the classification of vertex operator algebras generated by vectors of weight 2.
    \begin{definition}
The	{\bf  $W$-algebra $W(2,2)$} denoted by $\mathcal{W}$ is a complex Lie algebra that has a basis $\Big\{L_{m},W_{m},C_{1},C_{2}\mid m\in \mathbb{Z}\Big\}$ subject to the following Lie brackets:
	\begin{align*}
		&[L_{m},L_{n}]=(n-m)L_{m+n}+\frac{m^{3}-m}{12}\delta_{m+n,0}C_{1},\\
		&[L_{m},W_{n}]=(n-m)W_{m+n}+\frac{m^{3}-m}{12}\delta_{m+n,0}C_{2},\\
	&[W_{m},W_{n}]=[\mathcal{W},C_1]=[\mathcal{W},C_2]=0.
	\end{align*}
      \end{definition}

     For any $k\in\mathbb{N}$ and $n\in\mathbb{Z}$, define polynomials in $\mathbb{C}[t]$ as
     \begin{align*}h_{n,k;\alpha}(t)=nt^{k}-n(n-1)\alpha\frac{t^{k}-\alpha^{k}}{t-\alpha}.\end{align*}
 Let $\mathcal{T}_{\alpha}$ be the set of families of polynomials given by
 \begin{align*}
     \mathcal{T}_{\alpha}=\Big\{\{h_{m}(t)\mid m\in\mathbb{Z}\}\mid h_{m}(t)=\sum_{i=0}^{\infty}\xi_{i}h_{m,i;\alpha}(t),\xi_{i}\in\mathbb{C}\Big\}.
 \end{align*}
Note that, for a family $\mathbf h=\{h_m(t)\mid m\in\mathbb Z\}\in \mathcal{T}_{\alpha}$, we have
\begin{equation}\label{eq:hm-h1-relation}
h_m(t)=mh_1(t)-m(m-1)\alpha\frac{h_1(t)-h_1(\alpha)}{t-\alpha}.
\end{equation}
	 For $\alpha=0$, we have
	 $
	 \mathcal{T}_{0}=\Bigl\{\{mh(t)\}_{m\in\mathbb Z}\mid h(t)\in\mathbb C[t]\Bigr\}.
 $
 We recall the following two main results on $\mathcal{W}$ from \cite{CG17}.
 \begin{proposition}\label{pro2.2}
     The polynomial algebra $\Omega(\lambda,\alpha,\mathbf{h})=\mathbb{C}[t,s]$ is a $\mathcal{W}$-module for any $\alpha\in\mathbb{C},\lambda\in\mathbb{C}^{*}$ and $\mathbf{h}=\{h_{m}(t)\mid m\in\mathbb{Z}\}\in\mathcal{T}_{\alpha}$, with the following action:
     \begin{align*}
         &C_{1}f(t,s)=C_{2}f(t,s)=0,\  W_{m}f(t,s)=\lambda^{m}(t-m\alpha)f(t,s-m),\\
      &   L_{m}f(t,s)=\lambda^{m}(s+h_{m}(t))f(t,s-m)-m\lambda^{m}(t-m\alpha)\frac{\partial}{\partial t}f(t,s-m).
     \end{align*}
      Moreover, $\Omega(\lambda,\alpha,\mathbf{h})$ is simple if and only if $\alpha\neq0$.

 \end{proposition}
 \begin{theorem}\label{th2.3}
 Let $M$ be a $\mathcal W$-module that is free as a
$U(\mathbb C L_0\oplus\mathbb C W_0)$-module of rank $1$. Then $M$ is isomorphic to $\Omega(\lambda,\alpha,\mathbf{h})$ given  in Proposition \ref{pro2.2}.
 \end{theorem}
 \section{Modules over
 the Takiff Block type Lie algebra}
 We first introduce the Takiff Block type Lie algebra.
 \subsection{Takiff Block type Lie algebra}
\begin{definition}\label{de3.1}
 For any nonzero complex number $q$, the {\bf Takiff Block type Lie algebra} $\mathcal{B}(q)$ has a basis $\big\{L_{m,i},W_{m,i},C_{1},C_{2}\ |\ (m,i)\in \mathbb{Z}\times\mathbb{N}\big\}$ subject to the following relations
 \begin{align*}
     &[L_{m,i},L_{n,j}]=(n(i+q)-m(j+q))L_{m+n,i+j}+\frac{m^{3}-m}{12}\delta_{i+j,0}\delta_{m+n,0}C_{1},\\&
		[L_{m,i},W_{n,j}]=(n(i+q)-m(j+q))W_{m+n,i+j}+\frac{m^{3}-m}{12}\delta_{i+j,0}\delta_{m+n,0}C_{2},\\&
		[W_{m,i},W_{n,j}]=[\mathcal{B}(q),C_1]=[\mathcal{B}(q),C_2]=0.
 \end{align*}
 \end{definition}
 The elements $C_1$ and $C_2$ are central in $\mathcal B(q)$. If $-q\in\mathbb N$, then $L_{0,-q}$ and $W_{0,-q}$ are also central.
 Throughout this paper, we always denote $\mathfrak{h}=\mathbb{C}L_{0,0}\oplus\mathbb{C}W_{0,0}$. Let $\mathcal F_1(\mathcal B(q))$ denote the full subcategory of $\mathcal B(q)$-Mod whose objects are free of rank $1$ as $U(\mathfrak h)$-modules.

Set $\overline{\mathcal B}(q)=\mathcal B(q)/\operatorname{span}\{C_1,C_2\}$ and use the same symbols for the images of the generators. We record the following subalgebras of $\mathcal B(q)$ or $\overline{\mathcal B}(q)$.

 \begin{itemize}
\item[\rm(1)] For $q=-1$, the \textbf{Heisenberg--Virasoro algebra} is
\begin{align*}
\mathcal H=\operatorname{span}\{-L_{m,0},\,L_{m,1}\mid m\in\mathbb Z\}\cong\operatorname{span}\{-L_{m,0},\,W_{m,1}\mid m\in\mathbb Z\}
  \subseteq\overline{\mathcal B}(-1).
\end{align*}
\item[\rm(2)] the  {\bf $W$-algebra $W(2,2)$ }$\mathcal{W}= \text{span}\{q^{-1}L_{m,0},q^{-1}W_{m,0},C_{1},C_{2}\,|\,m\in\mathbb{Z}\}$;
\item[\rm(3)] For $q=-1$, the {\bf extended $W(2,2)$ algebra }$\mathcal{D}=\operatorname{span}\{-L_{m,0},\,W_{m,0},\,W_{m,1}\,|\, m\in\mathbb Z\}\subseteq\overline{\mathcal B}(-1)$.
\end{itemize}

For a family $\mathbf h(\alpha)=\{h_m(t)\mid m\in\mathbb Z\}\in \mathcal{T}_{\alpha}$, $c\in  \mathbb{C}$, set
\begin{align*}
\Theta_{\mathbf h,c}(t)
=c\frac{h_1(-t)-h_1(\alpha)}{-t-\alpha}\in\mathbb C[t].
\end{align*}
In the following, we always denote $\beta=(\beta_{1},\beta_{2})$ and $\gamma=(\gamma_{1},\gamma_{2})$.
 For any $\lambda\in\mathbb{C}^*$, $\alpha\in\mathbb{C}$, $\beta,\gamma\in \mathbb{C}^{2}$, $\mathbf h(\alpha)\in \mathcal{T}_{\alpha}$,
define the action of $\mathcal{B}(q)$ on $\Omega\big(\lambda,\alpha,\beta,\gamma,\mathbf h(\alpha)\big)=\mathbb C[t,s]$ by
     \begin{align}
C_{1}f(t,s)&=C_{2}f(t,s)=0,\label{C3.1}\\
L_{m,i}f(t,s)
&=\delta_{i,0}\lambda^{m}(s+qh_m(q^{-1}t))
f(t,s-mq)\label{L3.2}\\
&\quad-\delta_{i,0}m\lambda^{m}q(t-mq\alpha)
\frac{\partial}{\partial t}f(t,s-mq)
+\delta_{m,0}\delta_{i,-2q}\gamma_2f(t,s)\notag\\
&\quad+\delta_{i,1}\delta_{q,-1}\lambda^m
\left(\beta_2+m\Theta_{\mathbf h,\beta_1}(t)+m\beta_1
\frac{\partial}{\partial t}\right)
f(t,s+m),
\notag\\
W_{m,i}f(t,s)
&=\delta_{i,0}\lambda^m(t-mq\alpha)f(t,s-mq)
+\delta_{m,0}\delta_{i,-2q}\gamma_1f(t,s)\label{W3.3}\\
&\quad+\delta_{i,1}\delta_{q,-1}\beta_1\lambda^m
f(t,s+m),\notag
\end{align}
where $(m,i)\in \mathbb{Z}\times \mathbb{N}$.
In particular, when $q=\frac1n$ for some $n\in\mathbb Z_{+}$, for $\mathbf{h}(0)=\{h_{m}(t)\mid m\in\mathbb{Z}\}\in\mathcal{T}_{0}$ and $a,b\in\mathbb{C}$, define the action of $\mathcal{B}(q)$ on $\Omega_n(\lambda,a,b,\mathbf h(0))=\mathbb C[t,s]$ by
\begin{align}
C_1 f(t,s)&=C_2 f(t,s)=0,\notag \\
L_{m,i}f(t,s)
&=
\lambda^m a^i t^{in}
\left((1+in)s+\frac1n h_m(nt)-mi\right)
f\left(t,s-\frac mn\right) \label{eq:Omega-n-L}\\
&\quad
+ib\lambda^m a^{i-1} t^{1+in}
f\left(t,s-\frac mn\right)
-\frac mn\lambda^m a^i t^{1+in}
\frac{\partial}{\partial t}
f\left(t,s-\frac mn\right),\notag
\\
W_{m,i}f(t,s)
&=
\lambda^m a^i t^{1+in}
f\left(t,s-\frac mn\right),
\label{eq:Omega-n-W}
\end{align}
where $(m,i)\in\mathbb Z\times\mathbb N$. 
 This module can be regarded as an additional family occurring in the case $q=\frac1n$ and $\alpha=0$.

\begin{proposition}\label{prop3.1}

For any $\lambda,q\in\mathbb{C}^*$, $\alpha,a,b\in\mathbb C$, $\beta,\gamma\in \mathbb{C}^2$ and $\mathbf{h}(\alpha)\in\mathcal{T}_{\alpha}$, the following statements hold.
\item[{\rm (1)}]
$\Omega\big(\lambda,\alpha,\beta,\gamma,\mathbf h(\alpha)\big)\in \operatorname{Ob}(\mathcal F_1(\mathcal B(q)))$ under the actions of
\eqref{C3.1}--\eqref{W3.3}.
\item[{\rm (2)}] $\Omega_n\big(\lambda,a,b,\mathbf h(0)\big)\in\operatorname{Ob}(\mathcal F_1(\mathcal B(\frac{1}{n})))$ under the actions of \eqref{eq:Omega-n-L}--\eqref{eq:Omega-n-W}, where $n\in\mathbb Z_+$.
\end{proposition}
\begin{proof}
Since the actions of
$C_1$
  and
$C_2$ are trivial, we will omit them in the following proof.

{\rm (1)}
Since $\mathcal{B}(q)$ contains   $\mathcal{W}$ as a subalgebra, and the actions of $L_{m,0}$, $W_{m,0}$, $C_1$ and $C_2$ on $\Omega(\lambda,\alpha,\beta,\gamma,\mathbf h(\alpha))$ coincide with those in \cite{CG17}, the defining relations involving only the  $i=0$ part have already been verified.
The  nontrivial  action of $\mathcal{B}(q)$ on $\Omega(\lambda,\alpha,\beta,\gamma,\mathbf h(\alpha))$ in \eqref{C3.1}--\eqref{W3.3} can be written as
 \begin{align*}
      L_{m,i} f(t,s)=&
      \delta_{i,0}L_{m,0}f(t,s)
+\delta_{m,0}\delta_{i,-2q}\gamma_{2}f(t,s) \\
&+\delta_{i,1}\delta_{q,-1}\lambda^m
\left(\beta_2+m\Theta_{\mathbf h,\beta_1}(t)
+m\beta_1\frac{\partial}{\partial t}\right)f(t,s+m),
     \notag \\
      W_{m,i} f(t,s)=&
      \delta_{i,0}W_{m,0}f(t,s)
+\delta_{m,0}\delta_{i,-2q}\gamma_{1}f(t,s) +\delta_{i,1}\delta_{q,-1}\beta_{1}\lambda^{m}f(t,s+m).
     \end{align*}

If $q\neq -1$, then $L_{m,1}f(t,s)=W_{m,1}f(t,s)=0$ for all $m\in\mathbb Z$. It remains to check the defining relations involving $L_{0,-2q}$ and $W_{0,-2q}$. Since $L_{0,-2q}$ and $W_{0,-2q}$ act as scalars and all the remaining operators vanish, these relations follow directly from the definitions. Hence the formulas in \eqref{C3.1}-\eqref{W3.3} define a $\mathcal{B}(q)$-module in this case.

Now assume $q=-1$.
	According to \eqref{eq:hm-h1-relation} and $-\beta_1(h_m(-t)-mh_1(\alpha))=m(t+m\alpha)\Theta_{\mathbf h,\beta_1}(t)$, we have
	\begin{equation}\label{eq:theta-hm-derivative}
		-\beta_1\frac{d}{dt}h_m(-t)
		=m\bigl((t+m\alpha)\Theta_{\mathbf h,\beta_1}'(t)
		+\Theta_{\mathbf h,\beta_1}(t)\bigr).
	\end{equation}
		We now verify the relation
		$
		[L_{m,0},L_{n,1}]=-nL_{m+n,1}.
	$
		For any $f(t,s)\in \mathbb C[t,s]$, by \eqref{L3.2}, we have
		\begin{align*}
		L_{m,0}(L_{n,1}f(t,s))
		&=\lambda^{m+n}(s-h_m(-t))
		\bigl(\beta_2+n\Theta_{\mathbf h,\beta_1}(t)\bigr)f(t,s+m+n)\\
		&\quad+n\beta_1\lambda^{m+n}(s-h_m(-t))
		\frac{\partial f}{\partial t}(t,s+m+n)\\
		&\quad+mn\lambda^{m+n}(t+m\alpha)
		\Theta_{\mathbf h,\beta_1}'(t)f(t,s+m+n)\\
		&\quad+m\lambda^{m+n}(t+m\alpha)
		\bigl(\beta_2+n\Theta_{\mathbf h,\beta_1}(t)\bigr)
		\frac{\partial f}{\partial t}(t,s+m+n)\\
		&\quad+mn\beta_1\lambda^{m+n}(t+m\alpha)
		\frac{\partial^2f}{\partial t^2}(t,s+m+n).
		\end{align*}
		On the other hand,
		\begin{align*}
		L_{n,1}(L_{m,0}f(t,s))
		&=\lambda^{m+n}\bigl(\beta_2+n\Theta_{\mathbf h,\beta_1}(t)\bigr)
		(s+n-h_m(-t))f(t,s+m+n)\\
		&\quad+m\lambda^{m+n}\bigl(\beta_2+n\Theta_{\mathbf h,\beta_1}(t)\bigr)
		(t+m\alpha)\frac{\partial f}{\partial t}(t,s+m+n)\\
		&\quad+n\beta_1\lambda^{m+n}h_m'(-t)f(t,s+m+n)\\
		&\quad+n\beta_1\lambda^{m+n}(s+n-h_m(-t))
		\frac{\partial f}{\partial t}(t,s+m+n)\\
		&\quad+mn\beta_1\lambda^{m+n}\frac{\partial f}{\partial t}(t,s+m+n)\\
		&\quad+mn\beta_1\lambda^{m+n}(t+m\alpha)
		\frac{\partial^2f}{\partial t^2}(t,s+m+n).
		\end{align*}
		Therefore, by \eqref{eq:theta-hm-derivative}
		\begin{align*}
		[L_{m,0},L_{n,1}]f(t,s)
		&=-n\lambda^{m+n}\bigl(\beta_2+n\Theta_{\mathbf h,\beta_1}(t)\bigr)f(t,s+m+n)\\
		&\quad+n\lambda^{m+n}\bigl(m(t+m\alpha)\Theta_{\mathbf h,\beta_1}'(t)
		-\beta_1h_m'(-t)\bigr)f(t,s+m+n)\\
		&\quad-n(m+n)\beta_1\lambda^{m+n}\frac{\partial f}{\partial t}(t,s+m+n)\\
		&=-n\lambda^{m+n}\bigl(\beta_2+(m+n)\Theta_{\mathbf h,\beta_1}(t)\bigr)f(t,s+m+n)\\
		&\quad-n(m+n)\beta_1\lambda^{m+n}\frac{\partial f}{\partial t}(t,s+m+n)\\
		&=-nL_{m+n,1}f(t,s).
		\end{align*}
		Similarly, one can verify that
	$
		[L_{m,0},W_{n,1}]=-nW_{m+n,1}$,
		$[L_{m,1},W_{n,0}]=mW_{m+n,1}.
$  
It remains to verify $[L_{m,1},L_{n,1}]=0$. Since
\begin{align*}
    L_{m,1}L_{n,1}f(t,s)
    &=\lambda^{m+n}
    \left(\beta_2+m\Theta_{\mathbf h,\beta_1}(t)
    +m\beta_1\frac{\partial}{\partial t}\right)\\
	&\quad\cdot
	\left(\beta_2+n\Theta_{\mathbf h,\beta_1}(t)
	+n\beta_1\frac{\partial}{\partial t}\right)f(t,s+m+n)\\
	&=L_{n,1}L_{m,1}f(t,s).
\end{align*}
We obtain $[L_{m,1},L_{n,1}]=0$.

Finally, when $q=-1$, we have
$
	L_{0,2}=\gamma_2\,{\rm Id}$, $W_{0,2}=\gamma_1\,{\rm Id},
$
	and all the remaining $L_{m,i},W_{m,i}$ with $i\ge 3$ vanish. Hence all the other defining relations are immediate. Therefore the operators in \eqref{C3.1}-\eqref{W3.3} satisfy all defining relations of $\mathcal{B}(q)$, and thus $\Omega\big(\lambda,\alpha,\beta,\gamma,\mathbf h(\alpha)\big)$ is a $\mathcal{B}(q)$-module.

{\rm (2)}
		For $m,k\in\mathbb Z$, $i,j\in\mathbb N$, set
		\begin{align*}
		P_{m,i}(t,s)=(1+in)s+\frac1n h_m(nt)-mi,
		\qquad f_{m+k}=f\left(t,s-\frac{m+k}{n}\right).
		\end{align*}
		Since $\mathbf h(0)\in\mathcal T_0$, we have
		\begin{gather*}
		h_{m+k}(nt)=h_m(nt)+h_k(nt),\qquad mh_k'(nt)=kh_m'(nt),\\
		P_{k,j}\left(t,s-\frac mn\right)=P_{k,j}(t,s)-m\left(j+\frac1n\right).
		\end{gather*}
		For $f(t,s)\in\mathbb C[t,s]$,  we compute that 
		{\small
		\begin{align*}
		L_{m,i}(L_{k,j}f(t,s))
		&=\lambda^{m+k}a^{i+j}t^{(i+j)n}\\
		&\quad\cdot\left[P_{m,i}(t,s)P_{k,j}\left(t,s-\frac mn\right)
		-mjP_{k,j}\left(t,s-\frac mn\right)-\frac mn t h_k'(nt)\right]f_{m+k}\\
		&\quad+b\lambda^{m+k}a^{i+j-1}t^{1+(i+j)n}\\
		&\qquad\cdot\left[jP_{m,i}(t,s)+iP_{k,j}\left(t,s-\frac mn\right)-\frac{mj}{n}(1+jn)\right]f_{m+k}\\
		&\quad+ijb^2\lambda^{m+k}a^{i+j-2}t^{2+(i+j)n}f_{m+k}\\
		&\quad-\lambda^{m+k}a^{i+j}t^{1+(i+j)n}\\
		&\qquad\cdot\left[\frac knP_{m,i}(t,s)+\frac mnP_{k,j}\left(t,s-\frac mn\right)
		-\frac{mk}{n^2}(1+jn)\right]\frac{\partial f_{m+k}}{\partial t}\\
		&\quad-\frac{(ik+mj)b}{n}\lambda^{m+k}a^{i+j-1}t^{2+(i+j)n}
		\frac{\partial f_{m+k}}{\partial t}\\
		&\quad+\frac{mk}{n^2}\lambda^{m+k}a^{i+j}t^{2+(i+j)n}
		\frac{\partial^2f_{m+k}}{\partial t^2}.
		\end{align*}
		}
		On the other hand, we obtain
		{\small
		\begin{align*}
		L_{k,j}(L_{m,i}f(t,s))
		&=\lambda^{m+k}a^{i+j}t^{(i+j)n}\\
		&\quad\cdot\left[P_{k,j}(t,s)P_{m,i}\left(t,s-\frac kn\right)
		-kiP_{m,i}\left(t,s-\frac kn\right)-\frac kn t h_m'(nt)\right]f_{m+k}\\
		&\quad+b\lambda^{m+k}a^{i+j-1}t^{1+(i+j)n}\\
		&\qquad\cdot\left[iP_{k,j}(t,s)+jP_{m,i}\left(t,s-\frac kn\right)-\frac{ki}{n}(1+in)\right]f_{m+k}\\
		&\quad+ijb^2\lambda^{m+k}a^{i+j-2}t^{2+(i+j)n}f_{m+k}\\
		&\quad-\lambda^{m+k}a^{i+j}t^{1+(i+j)n}\\
		&\qquad\cdot\left[\frac mnP_{k,j}(t,s)+\frac knP_{m,i}\left(t,s-\frac kn\right)
		-\frac{mk}{n^2}(1+in)\right]\frac{\partial f_{m+k}}{\partial t}\\
		&\quad-\frac{(jm+ki)b}{n}\lambda^{m+k}a^{i+j-1}t^{2+(i+j)n}
		\frac{\partial f_{m+k}}{\partial t}\\
		&\quad+\frac{mk}{n^2}\lambda^{m+k}a^{i+j}t^{2+(i+j)n}
		\frac{\partial^2f_{m+k}}{\partial t^2}.
		\end{align*}
		}
		Subtracting the two expressions gives
		\begin{align*}
		[L_{m,i},L_{k,j}]f(t,s)
		&=\left(k\left(i+\frac1n\right)-m\left(j+\frac1n\right)\right)\\
		&\Bigg(\lambda^{m+k}a^{i+j}t^{(i+j)n}P_{m+k,i+j}(t,s)f_{m+k}\\
		& +(i+j)b\lambda^{m+k}a^{i+j-1}t^{1+(i+j)n}f_{m+k}\\
		&-\frac{m+k}{n}\lambda^{m+k}a^{i+j}t^{1+(i+j)n}
		\frac{\partial f_{m+k}}{\partial t}\Bigg)\\
		&=\left(k\left(i+\frac1n\right)-m\left(j+\frac1n\right)\right)
		L_{m+k,i+j}f(t,s).
		\end{align*}
			Similarly, we can get
			\begin{align*}
			[L_{m,i},W_{k,j}]f(t,s)
			=\left(k\left(i+\frac1n\right)-m\left(j+\frac1n\right)\right)
			W_{m+k,i+j}f(t,s)
			\end{align*}
		and
		\begin{align*}
		W_{m,i}(W_{k,j}f(t,s))
		=\lambda^{m+k}a^{i+j}t^{2+(i+j)n}f_{m+k}
		=W_{k,j}(W_{m,i}f(t,s)).
		\end{align*}
		Therefore all defining relations of $\mathcal{B}(\frac{1}{n})$ are satisfied, and
		$\Omega_n(\lambda,a,b,\mathbf h(0))$ is a $\mathcal{B}(\frac{1}{n})$-module.
\end{proof}
 \begin{theorem}\label{thm3.5}
    For any $\lambda\in\mathbb{C}^*$, $\alpha\in\mathbb{C}$, $\beta=(\beta_1,\beta_2)\in\mathbb{C}^{2}$, $\gamma=(\gamma_1,\gamma_2)\in \mathbb{C}^{2}$ and $\mathbf{h}(\alpha)\in \mathcal T_\alpha$, $\Omega(\lambda,\alpha,\beta,\gamma,\mathbf h(\alpha))$ is simple if and only if $\alpha\neq 0$ or $q=-1$ and $\beta_{1}\neq 0$.
 \end{theorem}
 \begin{proof}
Since $\mathcal{W}$ is a subalgebra of $\mathcal{B}(q)$, by Proposition 3.1 in \cite{CG17}, the module is simple whenever $\alpha \neq 0$.

Now assume that $q=-1$, $\alpha=0$ and $\beta_{1}\neq 0$. Suppose that $ N $ is a nonzero submodule of $\Omega(\lambda,\alpha,\beta,\gamma,\mathbf h(\alpha))$. Let $f(t,s)$ be a nonzero polynomial in $N$ of minimal degree in $s$. By \eqref{L3.2}-\eqref{W3.3}, we have
$f(t,s+m)\in N$ for all $m\in\mathbb{Z}$. This implies that $\deg_{s}(f(t,s))=0$, since otherwise $f(t,s)-f(t,s+m)\in N$ for $m\neq 0$ would have smaller degree in $s$ than $f(t,s)$, contradicting the minimality. Hence $f(t,s)=f(t)\in \mathbb{C}[t]$. From the action of $W_{0,0}$ and $L_{0,0}$, it follows that $f(t)\mathbb{C}[t,s]\subseteq N$. For $m\neq0$,
\begin{align*}
L_{m,1}f(t)-\lambda^m\bigl(\beta_2+m\Theta_{\mathbf h,\beta_1}(t)\bigr)f(t)
=m\lambda^m\beta_1f'(t)\in N.
\end{align*}
Then $f'(t)\in N$. This immediately gives that $N=\mathbb{C}[t,s]$, which means that $\Omega(\lambda,\alpha,\beta,\gamma,\mathbf h(\alpha))$ is simple.

If $\alpha=\beta_{1}=0$, then $t^{i}\mathbb{C}[t,s]$ is a submodule of $\Omega(\lambda,\alpha,\beta,\gamma,\mathbf h(\alpha))$ for every $i\in\mathbb{N}$.
 \end{proof}
We adopt the following shorthand notation. If
$-2q\notin\mathbb N$, then all terms involving $\beta$ and $\gamma$
vanish, and we write
\begin{align*}
\Omega(\lambda,\alpha,\mathbf h)
:=\Omega(\lambda,\alpha,0,0,\mathbf h).
\end{align*}
If $-2q\in\mathbb N$ and $q\neq -1$, then the only additional
parameters are $\gamma=(\gamma_1,\gamma_2)$, and we write
\begin{align*}
\Omega(\lambda,\alpha,\gamma,\mathbf h)
:=\Omega(\lambda,\alpha,0,\gamma,\mathbf h).
\end{align*}
 \begin{theorem}
    \label{prop3.2}
     $\Omega(\lambda,\alpha,\mathbf{h})\cong \Omega(\lambda_{1},\alpha_{1},\mathbf{h}_{1})$ if and only if $\lambda=\lambda_{1},\alpha=\alpha_{1}$ and $\mathbf{h}=\mathbf{h}_{1}$. $\Omega(\lambda,\alpha,\gamma,\mathbf{h})\cong \Omega(\lambda_{1},\alpha_{1},\gamma^{'},\mathbf{h}_{1})$ if and only if $\lambda=\lambda_{1},\alpha=\alpha_{1},\gamma=\gamma^{'}$ and $\mathbf{h}=\mathbf{h}_{1}$. $\Omega(\lambda,\alpha,\beta,\gamma,\mathbf h(\alpha))\cong \Omega(\lambda_{1},\alpha_1,\beta^{'},\gamma^{'},\mathbf h_{1}(\alpha_1))$ if and only if $\lambda=\lambda_{1},\alpha=\alpha_{1},\gamma=\gamma^{'},\beta=\beta^{'}$ and $\mathbf{h}=\mathbf{h}_{1}$.
 \end{theorem}
\begin{proof}
    According to \cite{CG17}, we have $\Omega(\lambda,\alpha,\mathbf{h})\cong \Omega(\lambda_{1},\alpha_{1},\mathbf{h}_{1})$ if and only if $\lambda=\lambda_{1},\alpha=\alpha_{1}$ and $\mathbf{h}=\mathbf{h}_{1}$. Following a similar argument as in the proof of Proposition 3.2 in \cite{CG17}, we deduce that if $\Omega(\lambda,\alpha,\gamma,\mathbf{h})\cong \Omega(\lambda_{1},\alpha_{1},\gamma^{'},\mathbf{h}_{1})$ or
$\Omega(\lambda,\alpha,\beta,\gamma,\mathbf h(\alpha))\cong \Omega(\lambda_{1},\alpha_1,\beta{'},\gamma^{'},\mathbf h_{1}(\alpha_1))$, then it must hold that $\lambda=\lambda_{1},\alpha=\alpha_{1}$ and
    $\mathbf{h}=\mathbf{h}_{1}$.

    Let $\phi$ be the isomorphism from $\Omega(\lambda,\alpha,\gamma,\mathbf{h})$ to $ \Omega(\lambda_{1},\alpha_{1},\gamma^{'},\mathbf{h}_{1})$.

    Then, from
    \begin{align*}
    \phi(W_{0,-2q}f(t,s))=\phi(\gamma_{1}f(t,s))=\gamma_{1}\phi(f(t,s))=W_{0,-2q}\phi(f(t,s))=\gamma_{1}^{'}\phi(f(t,s)),
    \end{align*}
    we obtain $\gamma_{1}=\gamma_{1}^{'}$. Similarly, applying $\phi(L_{0,-2q}f(t,s))=L_{0,-2q}\phi(f(t,s))$ yields $\gamma_{2}=\gamma_{2}^{'}$.

    Next, let $\psi$ be the isomorphism from $\Omega(\lambda,\alpha,\beta,\gamma,\mathbf h(\alpha))$ to $\Omega(\lambda_{1},\alpha_1,\beta^{'},\gamma^{'},\mathbf h_{1}(\alpha_1))$. From $\psi(W_{0,1}f(t,s))=W_{0,1}\psi(f(t,s))$ and $\psi(L_{0,1}f(t,s))=L_{0,1}\psi(f(t,s))$, we conclude that $\beta_{1}=\beta_{1}^{'}$ and $\beta_{2}=\beta_{2}^{'}$. Since $\mathbf h=\mathbf h_1$ and $\beta_1=\beta_1'$, the corresponding polynomials $\Theta_{\mathbf h,\beta_1}(t)$ are also equal.
\end{proof}

\subsection{$U(\mathfrak{h})$-free modules of rank one}
We denote $\mathfrak{h}=\mathbb{C}L_{0,0}\oplus\mathbb{C}W_{0,0}$. In this subsection, we will give a complete classification of $U(\mathfrak{h})$-free modules of rank $1$ over $\mathcal{B}(q)$.

Let $M\in\operatorname{Ob}(\mathcal F_1(\mathcal B(q)))$.
It follows from Definition \ref{de3.1} that
$
[L_{0,0},W_{0,0}]=0,
$
and we have $U(\mathfrak{h})=\mathbb{C}[L_{0,0},W_{0,0}]$. Choose a generator $1$ of $M$. Then
\begin{align*}
M=U(\mathfrak h) 1=\mathbb C[L_{0,0},W_{0,0}] 1.
\end{align*}
Write
$
s=L_{0,0} 1,\  t=W_{0,0} 1.
$
Then $M$ is naturally identified with $\mathbb C[t,s]$.
It is clear that $M$ can be seen as a $\mathcal{W}$-module.
    For any $m\in\mathbb{Z}$, Theorem~\ref{th2.3} gives
    \begin{equation}\label{action}
    \begin{aligned}
    C_{1} f(t,s)&=C_{2} f(t,s)=0,
    W_{m,0} f(t,s)
     =\lambda^{m}(t-mq\alpha)f(t,s-mq),\\
    L_{m,0} f(t,s)
    &=\lambda^{m}\bigl(s+q h_{m}(q^{-1}t)\bigr)f(t,s-mq)  -m\lambda^{m}q(t-m q\alpha)\frac{\partial}{\partial t}f(t,s-mq).
    \end{aligned}
    \end{equation}
    In the following proof, we omit the actions of $C_1$ and $C_2$.
 \begin{lemma}\label{lem3.2}
 For $(m,i)\in \mathbb{Z}\times \mathbb{N}$ and $f(t,s)\in\mathbb{C}[t,s]$, we have
 \begin{align}
        W_{m,i} f(t,s)=&f(t,s-mq)W_{m,i}1,\label{3.2}\\
        L_{m,i} f(t,s)=&f(t,s-mq)L_{m,i} 1 -mq\frac{\partial}{\partial t}f(t,s-mq)W_{m,i}1.\label{3.3}
 \end{align}
    In particular, if $[L_{m_{0},i_{0}},W_{0,0}] 1=0$ for some  $(m_{0},i_{0})\in \mathbb{Z}\times \mathbb{N}$, then $L_{m_{0},i_{0}} f(t,s)=f(t,s-m_{0}q)L_{m_{0},i_{0}}1$.
 \end{lemma}
 \begin{proof}
     Since $[W_{m,i},W_{0,0}] f(t,s)=0$ and $[W_{m,i},L_{0,0}] f(t,s)=-mqW_{m,i} f(t,s)$, we obtain
     \begin{align*}
         W_{m,i} tf(t,s)=tW_{m,i} f(t,s),\
         W_{m,i} sf(t,s)=(s-mq)W_{m,i} f(t,s).
     \end{align*}
By induction on $a,b\in \mathbb N$, we conclude
 $
W_{m,i}(t^as^b)=t^a(s-mq)^b\,W_{m,i}1.$
Then for any $f(t,s)\in \mathbb C[t,s]$, we see that
$
W_{m,i}f(t,s)=f(t,s-mq)\,W_{m,i} 1,
$
which proves \eqref{3.2}.

We now prove \eqref{3.3}. From
$
[L_{m,i},W_{0,0}]=-mq\,W_{m,i}$ and $ [L_{m,i},L_{0,0}]=-mq\,L_{m,i},
$
 we deduce that
\begin{eqnarray}\label{L3.13}
L_{m,i}\bigl(t f(t,s)\bigr)=t\,L_{m,i}f(t,s)-mq\,W_{m,i}f(t,s),
\end{eqnarray}
\begin{eqnarray}\label{W3.14}
L_{m,i}\bigl(s f(t,s)\bigr)=(s-mq)\,L_{m,i}f(t,s).
\end{eqnarray}

We first consider the case $f(s)\in \mathbb C[s]$. By \eqref{W3.14}, we obtain
$
L_{m,i}f(s)=f(s-mq)\,L_{m,i}  1.$
Since in this case $\frac{\partial}{\partial t}f(s-mq)=0,$
formula \eqref{3.3} holds for all $f(s)\in \mathbb C[s]$.
Next, we prove \eqref{3.3} for general $f(t,s)\in \mathbb C[t,s]$ by induction on the degree of $f$ in $t$. Assume that
\begin{align*}
L_{m,i}f(t,s)=f(t,s-mq)L_{m,i}  1-mq\, \frac{\partial}{\partial t}f(t,s-mq)W_{m,i}1
\end{align*}
holds for some $f(t,s)\in \mathbb C[t,s]$. Then, using \eqref{L3.13} together with \eqref{3.2}, we get
\begin{align*}
L_{m,i}\bigl(t f(t,s)\bigr)
&=t\,L_{m,i}f(t,s)-mq\,W_{m,i}f(t,s)\\
&=t f(t,s-mq)L_{m,i}1\\
&\quad-mq\,t\frac{\partial}{\partial t}f(t,s-mq)W_{m,i}1
-mq\,f(t,s-mq)W_{m,i}1 \\
&=t f(t,s-mq)L_{m,i} 1
   -mq\, \Bigl(t\frac{\partial}{\partial t}f(t,s-mq)+f(t,s-mq)\Bigr)W_{m,i}1\\
&=t f(t,s-mq)L_{m,i}  1
   -mq\, \frac{\partial}{\partial t}\bigl(t f(t,s-mq)\bigr)W_{m,i}1.
\end{align*}
Thus \eqref{3.3} holds for $tf(t,s)$. By induction on the degree of $f$ in $t$, we conclude that
$
L_{m,i}f(t,s)=f(t,s-mq)L_{m,i} 1 -mq\,\frac{\partial}{\partial t}f(t,s-mq)W_{m,i} 1
$
for all $f(t,s)\in \mathbb C[t,s]$, which proves \eqref{3.3}.
 \end{proof}
For $(m,i)\in\mathbb{Z}\times\mathbb{N}$, we write
\begin{align}\label{WL3.15}
W_{m,i}  1=g_{m,i}(t,s)\in\mathbb{C}[t,s],\ L_{m,i} 1 =f_{m,i}(t,s)\in\mathbb{C}[t,s]
\end{align}
in the following.  Then it follows from \eqref{action} that  $g_{m,0}(t,s)=\lambda^{m}(t-mq\alpha)$ and $f_{m,0}(t,s)=\lambda^{m}\bigl(s+q h_{m}(q^{-1}t)\bigr)$. Clearly, $g_{0,0}(t,s)=t$ and $f_{0,0}(t,s)=s$.

\subsubsection{The polynomials $g_{m,i}(t,s)$}
In this subsection, we shall determine $g_{m,i}(t,s)$  of \eqref{WL3.15} for all $(m,i)\in\mathbb{Z}\times\mathbb{N}$.
 \begin{lemma}\label{lem3.3}
     For $(n,i)\in\mathbb{Z}\times\mathbb{N}$, we obtain $g_{n,i}(t,s)=g_{n,i}(t)$.
 \end{lemma}
 \begin{proof}
 By \eqref{action} and Lemma \ref{lem3.2}, for $m\in\mathbb{Z}^*,(n,i)\in\mathbb{Z}\times\mathbb{N}$, we have
     \begin{align*}
         0=&[W_{m,0},W_{n,i}]  1\\
         =&\lambda^{m}(t-mq\alpha)g_{n,i}(t,s-mq)-\lambda^{m}(t-mq\alpha)g_{n,i}(t,s)\\
         =&\lambda^{m}(t-mq\alpha)(g_{n,i}(t,s-mq)-g_{n,i}(t,s)).
     \end{align*}
     Since $q\neq0$, the above identity implies that $g_{n,i}(t,s)$ is independent of $s$, so $g_{n,i}(t,s)=g_{n,i}(t)$.
 \end{proof}
\begin{lemma}\label{lem:basic-W-equations}
If $m(i+q)\neq0$, we have
\begin{align}
    \label{eq:gmi-from-g0i}
g_{m,i}(t)
&=
\frac{q\lambda^m  (t-mq\alpha)}{i+q}\,g'_{0,i}(t),
\\ \label{eq:g0i-differential}
(i+2q)(i+q)g_{0,i}(t)
&=
2q^2t\,g'_{0,i}(t)
+
q^2(t^2-(mq\alpha)^2)g''_{0,i}(t).
\end{align}
In particular, if $-2q\in\mathbb N$, then $g_{0,-2q}(t)\in\mathbb C$.
\end{lemma}
\begin{proof}
From
$[W_{0,i},L_{m,0}]=m(i+q)W_{m,i}$, we obtain
\begin{align*}
m(i+q)W_{m,i}1
=
W_{0,i}L_{m,0}1-L_{m,0}W_{0,i}1.
\end{align*}
By \eqref{action} and Lemma~\ref{lem3.3}, we have
\begin{align*}
m(i+q)g_{m,i}(t)
&=
W_{0,i}\Bigl(\lambda^m\bigl(s+qh_m(q^{-1}t)\bigr)\Bigr)
-
L_{m,0}g_{0,i}(t)\\
&=
\lambda^m\bigl(s+qh_m(q^{-1}t)\bigr)g_{0,i}(t)
-
\lambda^m\bigl(s+qh_m(q^{-1}t)\bigr)g_{0,i}(t)\\
&\quad
+
mq\lambda^m(t-mq\alpha)g'_{0,i}(t)\\
&=
mq\lambda^m(t-mq\alpha)g'_{0,i}(t).
\end{align*}
Considering $m(i+q)\neq0$, we have
$
g_{m,i}(t)
=
\frac{q\lambda^m  (t-mq\alpha)}{i+q}\,g'_{0,i}(t),
$
which gives \eqref{eq:gmi-from-g0i}.
Then by
\eqref{action}  and
$[L_{-m,0},W_{m,i}]=m(i+2q)W_{0,i}$, we check that
 \begin{align*}
     m(i+2q)g_{0,i}(t)=&[L_{-m,0},W_{m,i}] 1\\
     =&L_{-m,0} g_{m,i}(t)-W_{m,i} L_{-m,0} 1\\
     =&\frac{2mq^2t}{i+q}g^\prime_{0,i}(t)+\frac{mq^2(t^2-(mq\alpha)^2)}{i+q}g^{\prime\prime}_{0,i}(t),
 \end{align*}
 which shows
 \eqref{eq:g0i-differential}.

Finally, suppose that $-2q\in\mathbb N$. Taking $i=-2q$ in
\eqref{eq:g0i-differential}, we obtain
\begin{align*}
2t\,g'_{0,-2q}(t)
+
(t^2-(mq\alpha)^2)g''_{0,-2q}(t)=0.
\end{align*}
Write
$
g_{0,-2q}(t)=\sum_{k=0}^{N}b_kt^k,$ where $ b_N\neq0.
$
Comparing the coefficient of $t^N$, we get
\begin{align*}
N(N+1)b_N=0,
\end{align*}
which implies $g_{0,-2q}(t)=\gamma_{1}$, where $\gamma_{1}\in\mathbb{C}$.
This completes the proof.
\end{proof}

\begin{lemma}\label{lem:ordinary-W-vanishing}
Suppose that $q\neq -1, \frac1n$ for any
$n\in\mathbb Z_{+}$. Then
$g_{m,i}(t)=0$
for all $(m,i)\in\mathbb Z\times\mathbb Z_{+}$ with $(m,i)\neq(0,-2q)$.
\end{lemma}
\begin{proof}
Assuming $g_{0,i}(t)\neq 0$, we can write
$
g_{0,i}(t)=\sum_{k=0}^{N}a_{k}t^k$ for  $a_N\neq0$.
Consider $m(i+q)\neq0$. Taking the coefficient of $t^N$ in \eqref{eq:g0i-differential}, we get
$
(i+2q)(i+q)a_N
=
q^2N(N+1)a_N,
$
which implies
\begin{align}\label{eq:N}
    N=1+\frac{i}{q}
\quad\text{or}\quad
N=-2-\frac{i}{q}.
\end{align}
We now consider the following three cases:

{\bf Case 1}.  $|\frac{1}{q}|\notin \mathbb{N}$.

By \eqref{eq:N}, we get $g_{0,1}(t)=0$. Since $[L_{m,i},W_{0,1}] 1=-m(1+q)W_{m,i+1} 1$ and $q\neq -1$, it follows that $g_{m,i}(t)=0$ for all $(m,i)\in \mathbb{Z}^{*}\times\mathbb Z_{+}$. From $[L_{1,0},W_{-1,i}] 1=-(i+2q)W_{0,i} 1$, we obtain $g_{0,i}(t)=0$ for $i\in\mathbb Z_{+}\setminus\{-2q\}$.

{\bf Case 2}. $q=-\frac{1}{2}$.

By Lemma~\ref{lem:basic-W-equations}, we have
$g_{0,1}(t)=\gamma_1\in\mathbb C.$
Since $g'_{0,1}(t)=0$, it follows from
\eqref{eq:gmi-from-g0i} that
$g_{m,1}(t)=0$
for all $m\in\mathbb Z^*$.
For   $(m,i)\in\mathbb Z^*\times\mathbb N$, applying
$[L_{m,i},W_{0,1}]1=-m(1+q)W_{m,i+1}1,$
we obtain
$g_{m,j}(t)=0$ for all $m\in\mathbb Z^*$ and $j\geq2$.
From
$[L_{1,0},W_{-1,i}]1=-(i+2q)W_{0,i}1$,
we get
$(i-1)g_{0,i}(t)=0,$
which gives
$g_{0,i}(t)=0$
for all $i\in\mathbb{Z}_{+}\setminus\{1\}$.

{\bf Case 3}. $q\in \{-\frac{1}{n}\mid n\in\mathbb{N}\} \setminus \{-\frac{1}{2},-1\}$.

In this case, we have $q=-\frac1n$ for some $n\geq3$.
Taking $i=1$ in \eqref{eq:N}, if $g_{0,1}(t)\neq0$, then
$
\deg g_{0,1}(t)=-2-\frac1q=n-2.
$
Therefore, we get
$
g_{0,1}(t)=ct^{n-2}+\text{lower degree terms}$
for some $c\neq0.$
By \eqref{eq:gmi-from-g0i}, for $m\neq0$, we have
\begin{align*}
g_{m,1}(t)
=
-\frac{\lambda^m}{n-1}(t+\frac mn\alpha)g'_{0,1}(t),
\end{align*}
which shows the coefficient of $t^{n-2}$ in $g_{m,1}(t)$ is
$-\frac{\lambda^m(n-2)c}{n-1}.$

Now for any $(k,m)\in \mathbb{Z}^*\times\mathbb{Z}$, applying
$[L_{k,0},W_{m,1}]1
=
\left(-\frac{m}{n}-k(1-\frac{1}{n})\right)W_{m+k,1}1$,
 we get
\begin{align*}
 -\frac{\lambda^k}{n}\left(mg_{m,1}(t)-k(t-kq\alpha)g'_{m,1}(t)\right)
=
\left(mq-k(1+q)\right)g_{m+k,1}(t).
\end{align*}
Comparing the highest degree   of the above equation, we obtain
$k(n-2)=-k(n-1).$
Since $k\neq0$, this implies $2n-3=0$, a contradiction to $n\geq3$. Then we have
$g_{0,1}(t)=0.$
By the same argument   in Case 1,
we conclude that
$g_{m,j}(t)=0$
for all $(m,j)\in\mathbb Z^*\times\mathbb Z_{+}$.

Finally, from
$[L_{1,0},W_{-1,i}]1=-(i+2q)W_{0,i}1$
and $-2q=\frac2n\notin\mathbb N$ for $n\geq3$, it follows that
$g_{0,i}(t)=0$
for all $i\in\mathbb Z_{+}$. Therefore, we have $g_{m,i}(t)=0$ for all $(m,i)\in\mathbb Z\times\mathbb Z_{+}$ with $(m,i)\neq(0,-2q)$.
This proves the  lemma.
\end{proof}

\begin{lemma}\label{lem:q-minus-one-W-vanishing}
Suppose that $q=-1$. Then
$
g_{m,i}(t)=0
$
for all $i\in\mathbb Z_+\setminus\{1,2\}.$  Moreover,
\begin{align*}
g_{m,1}(t)=\lambda^m\beta_1, \
g_{0,2}(t)=\gamma_1, \
g_{m,2}(t)=0 \ (m\neq0),
\end{align*}
where $\beta_1,\gamma_1\in\mathbb C$.
\end{lemma}
\begin{proof}
By Lemma~\ref{lem:basic-W-equations}, we get
$g_{0,2}(t)=\gamma_1\in\mathbb C.$
Since $g'_{0,2}(t)=0$, it follows
from \eqref{eq:gmi-from-g0i} that
$g_{m,2}(t)=0$
for all $m\in\mathbb Z^*$.
Applying
$
[L_{m,i},W_{0,2}]1=-mW_{m,i+2}1,
$
we obtain
$g_{m,i+2}(t)=0$
for all $(m,i)\in\mathbb Z^*\times\mathbb N$. Hence
$g_{m,j}(t)=0$
for all $m\in\mathbb Z^*$ and $j\geq2$.
Then  by
$[L_{1,0},W_{-1,i}]1=(2-i)W_{0,i}1,$
we get
$g_{0,i}(t)=0$
for all $i\ge3$.
It remains to determine $g_{m,1}(t)$. Since
$
[L_{m,0},W_{0,1}]1=0,
$  we obtain
$
m\lambda^m(t+m\alpha)g'_{0,1}(t)=0.
$
Hence $g'_{0,1}(t)=0$,  that is to say,
$g_{0,1}(t)=\beta_1$
for some $\beta_1\in\mathbb C$.
Based on
$[L_{-m,0},W_{m,1}]1=-mW_{0,1}1$, we have
$
m\beta_1
=
m\lambda^{-m}g_{m,1}(t)
+
m\lambda^{-m}(t-m\alpha)g'_{m,1}(t).
$
If $m\neq 0$, we obtain
$
g_{m,1}(t)+(t-m\alpha)g'_{m,1}(t)=\lambda^m\beta_1.
$
Comparing the highest degree term
shows that $g_{m,1}(t)$ is constant. Therefore,
$
g_{m,1}(t)=\lambda^m\beta_1$ for all $m\in\mathbb{Z}$.
This completes the proof.
\end{proof}
\begin{lemma}\label{lem:positive-reciprocal-W-action}
Suppose that $q=\frac1n$ for some
$n\in\mathbb Z_+$. Then the following statements hold.

\begin{itemize}
\item If $\alpha\neq0$, then
$
g_{m,i}(t)=0
$
for all $(m,i)\in\mathbb Z\times\mathbb Z_+$.

\item If $\alpha=0$, then there exists $a\in\mathbb C$ such that
$
g_{m,i}(t)=\lambda^m a^i t^{1+in}
$
for all $(m,i)\in\mathbb Z\times\mathbb N$.
\end{itemize}
\end{lemma}

\begin{proof}
Let $i\in\mathbb Z_+$. Since $q=\frac1n$, we obtain $(i+kq) \neq0$ for $k\in\mathbb{Z}_+$.
From \eqref{eq:g0i-differential} and $m\neq0$, one has
\begin{eqnarray}\label{eq3.18}
(i+2q)(i+q)g_{0,i}(t)
=
2q^2t\,g'_{0,i}(t)
+
q^2t^2g''_{0,i}(t)
-
m^2q^4\alpha^2g''_{0,i}(t).
\end{eqnarray}
 Now we give the following two cases.

{\bf Case 1}. $\alpha\neq0$.

Considering the highest degree of $m$ in \eqref{eq3.18}, we  have
$
g''_{0,i}(t)=0,
$
which  shows $\deg g_{0,i}(t)\leq1$.  Then we can write $
g_{0,i}(t)=a_{1,i}t+a_{0,i}
$ for $i\in\mathbb{Z}_+,a_{1,i},a_{0,i}\in\mathbb{C}$. Substituting this into \eqref{eq3.18} gives $g_{0,i}(t)=0$ for all $i\in\mathbb{Z}_+$.
By \eqref{eq:gmi-from-g0i}, we get
$
g_{m,i}(t)=0
$
for all $(m,i)\in\mathbb Z\times\mathbb Z_+$.

{\bf Case 2}. $\alpha=0$.

\eqref{eq3.18} becomes
$
(i+2q)(i+q)g_{0,i}(t)
=
2q^2t\,g'_{0,i}(t)+q^2t^2g''_{0,i}(t).
$
Writing $g_{0,i}(t)=\sum_k c_kt^k$, comparison of each nonzero term gives
$
q^2k(k+1)=(i+2q)(i+q).
$
Again the only nonnegative solution is $k=1+\frac iq=1+in$. Then we get
$
g_{0,i}(t)=a_i t^{1+in}
$
for some $a_i\in\mathbb C$. By \eqref{eq:gmi-from-g0i}, we obtain
\begin{align}\label{eq:gmi-ai}
g_{m,i}(t)=\lambda^m a_i t^{1+in}
\end{align}
for all $(m,i)\in\mathbb{Z}\times\mathbb{N}$ and $a_0=1$.
Now we determine the coefficients $a_i$ for $i\in\mathbb{Z}_+$ in \eqref{eq:gmi-ai}.

If $a_1=0$, then
$g_{1,1}(t)=0$. Applying
$[L_{0,1},W_{1,i}]=(1+q)W_{1,i+1}$
to $1$, we have
\begin{align*}
(1+q)g_{1,i+1}(t)
=
g_{1,i}(t)\bigl(f_{0,1}(t,s)-f_{0,1}(t,s-q)\bigr).
\end{align*}
Thus $g_{1,2}(t)=0$, and then inductively $g_{1,i}(t)=0$ for all
$i\in\mathbb Z_+$. Hence $a_i=0$ for all $i\in \mathbb Z_+$, and the conclusion holds
with $a=0$.

Now assume $a_1\neq0$.
According to
$[L_{0,1},W_{1,1}]1=(1+q)W_{1,2}1$, we obtain
\begin{align*}
(1+q)a_2t^{1+2n}
=
a_1t^{1+n}\bigl(f_{0,1}(t,s)-f_{0,1}(t,s-q)\bigr).
\end{align*}
Then $f_{0,1}(t,s)-f_{0,1}(t,s-q)\in\mathbb C[t]$. Set
$
f_{0,1}(t,s)=A_{0,1}(t)s+B_{0,1}(t),
$ where $A_{0,1}(t),B_{0,1}\in\mathbb C[t]$.
Note that
$
f_{0,1}(t,s)-f_{0,1}(t,s-q)=qA_{0,1}(t).
$
Therefore,
\begin{align*}
A_{0,1}(t)=\frac{(1+q)a_2}{qa_1}t^n.
\end{align*}

For $i\in\mathbb N$, applying
$[L_{0,1},W_{1,i}]1=(1+q)W_{1,i+1}1$, we get
\begin{equation}\label{eq:ai-recursion}
a_{i+1}=\frac{a_2}{a_1}a_i.
\end{equation}

Next by using
$[L_{0,1},L_{1,0}]1=(1+q)L_{1,1}1$, we have
\begin{align}\label{eqf3.21}
    (1+q)f_{1,1}(t,s)=\lambda\big(s+qh_1(q^{-1}t)\big)\big(f_{0,1}(t,s)-f_{0,1}(t,s-q)\big)+\lambda qt\frac{\partial}{\partial t}f_{0,1}(t,s),
    \end{align}
which implies ${\rm deg}_s(f_{1,1}(t,s))\le 1$. Then we can write
$f_{1,1}(t,s)=A_{1,1}(t)s+B_{1,1}(t)$, where $A_{1,1}(t), B_{1,1}(t)\in\mathbb{C}[t]$.
  Comparing the coefficients of $s$ in \eqref{eqf3.21}, we obtain
$
A_{1,1}(t)
=
\lambda(n+1)\frac{a_2}{a_1}t^n.
$
On the other hand, from
$[L_{1,1},W_{1,1}]1=0$
  and   Lemma~\ref{lem3.2}, we see that
$
A_{1,1}(t)=g'_{1,1}(t).
$
By \eqref{eq:gmi-ai},
we have
$
A_{1,1}(t)=\lambda(n+1)a_1t^n.
$
Comparing the two expressions for $A_{1,1}(t)$, we get
$
a_2=a_1^2.
$
Hence by \eqref{eq:ai-recursion}, one has  $a_i=a_1^i$ for all $i\in\mathbb N$. Setting $a=a_1$, then \eqref{eq:gmi-ai} becomes
$
g_{m,i}(t)=\lambda^m a^i t^{1+in},$ where $ (m,i)\in\mathbb Z\times\mathbb N.
$
This proves the lemma.
\end{proof}

\subsubsection{The polynomials $f_{m,i}(t,s)$}
In this subsection, we shall determine $f_{m,i}(t,s)$  of \eqref{WL3.15} by the following lemmas for all $(m,i)\in\mathbb{Z}\times\mathbb{N}$.

\begin{lemma}\label{lem:ordinary-L-vanishing}
Suppose $q\neq -1$, and suppose that either
$q\neq \frac1n$ for all $n\in\mathbb Z_+$, or $\alpha\neq0$.
 Then
$
f_{m,i}(t,s)=0
$
for all $(m,i)\in\mathbb Z\times\mathbb Z_+$, except possibly
$(m,i)=(0,-2q)$. Moreover, if $-2q\in\mathbb Z_+$, then
$
f_{0,-2q}(t,s)=\gamma_2
$
for some $\gamma_2\in\mathbb C$.
\end{lemma}

\begin{proof}
By Lemmas \ref{lem:ordinary-W-vanishing} and \ref{lem:positive-reciprocal-W-action}, we know
  that
$
g_{m,i}(t)=0
$
for all $(m,i)\in\mathbb Z\times\mathbb Z_+$, except possibly
$(m,i)=(0,-2q)$.
From  Lemma~\ref{lem3.2}, for
$(m,i)\neq(0,-2q)$ we have
$
L_{m,i}r(t,s)=r(t,s-mq)f_{m,i}(t,s)$,
where $r(t,s)\in\mathbb C[t,s].
$

Now let $(m,i)\neq(0,-2q)$.   From   $[L_{m,i},W_{-m,0}]1=0$, we deduce that
$
\lambda^{-m}(t+mq\alpha)
\bigl(f_{m,i}(t,s)-f_{m,i}(t,s+mq)\bigr)=0,
$
which shows
\begin{align*}
f_{m,i}(t,s)=f_{m,i}(t),\qquad m\neq0.
\end{align*}
Applying $[L_{0,i},L_{m,0}]1=m(i+q)L_{m,i}1$, we check
\begin{align}\label{fmi3.23}
m(i+q)f_{m,i}(t)
=&\lambda^m\bigl(s+qh_m(q^{-1}t)\bigr)
\bigl(f_{0,i}(t,s)-f_{0,i}(t,s-mq)\bigr)  \\
&\nonumber
+mq\lambda^m(t-mq\alpha)
\frac{\partial}{\partial t}f_{0,i}(t,s-mq).
\end{align}
Since the left hand side is independent of $s$, comparing the highest degree of   $s$ shows that $f_{0,i}(t,s)\in\mathbb C[t]$. Based on   \eqref{fmi3.23}, we have for $i\neq -q$,
\begin{align}\label{eq:fmi}
f_{m,i}(t)
=
\frac{\lambda^mq(t-mq\alpha)}{i+q}f'_{0,i}(t).
\end{align}
Now it follows from  $[L_{-m,0},L_{m,i}]1=m(i+2q)L_{0,i}1$ that
\begin{align*}
(i+2q)(i+q)f_{0,i}(t)
=
2q^2t f'_{0,i}(t)
+
q^2\big(t^2-(mq\alpha)^2\big)f''_{0,i}(t).
\end{align*}
This is exactly the same differential equation as in
Lemma~\ref{lem:basic-W-equations}, with $g_{0,i}$ replaced by $f_{0,i}$.
Therefore the same highest degree argument used in
Lemmas~\ref{lem:ordinary-W-vanishing} and
\ref{lem:positive-reciprocal-W-action} gives
\begin{align*}
f_{0,i}(t)=0
\quad\text{for } i\notin\{-q,-2q\},
\end{align*}
and if $-2q\in\mathbb Z_+$,
\begin{align*}
f_{0,-2q}(t)=\gamma_2\in\mathbb C.
\end{align*}
Then by \eqref{eq:fmi},
$f_{m,i}(t)=0$ for all remaining $m\neq0$ with $i\neq-q$. 

It remains to consider $-1\neq i=-q\in\mathbb Z_+$. The above argument gives $L_{m,-q-1}=L_{m,1}=0$ for all $m\in\mathbb Z$. Since
\begin{align*}
[L_{m,-q-1},L_{n,1}]=\bigl((-q-1)m-n\bigr)L_{m+n,-q}
\end{align*}
we have $L_{m,-q}=0$ for every $m\in\mathbb Z$.
This completes the proof.
\end{proof}
\begin{lemma}\label{lem:positive-reciprocal-L-action}
Suppose that $q=\frac1n$ for some $n\in\mathbb Z_+$, and
$\alpha=0$.
Then there exists $b\in\mathbb C$ such that
\begin{align*}
f_{m,i}(t,s)
=
\lambda^m a^i t^{in}
\left((1+in)s+\frac1n h_m(nt)-mi\right)
+i b\lambda^m a^{i-1}t^{1+in}.
\end{align*} 
Here set $i b a^{i-1}=0$ for $i=0$.
\end{lemma}

\begin{proof}
By Lemma \ref{lem:positive-reciprocal-W-action}, we have
$
g_{m,i}(t)=\lambda^m a^i t^{1+in}$,
 where   $(m,i)\in\mathbb Z\times\mathbb N$  and $a\in\mathbb{C}.
$
According to
$
[L_{m,i},W_{1,0}]1=(i+\frac{1-m}{n})W_{m+1,i}1
$, we obtain
\begin{align*}
f_{m,i}(t,s)-f_{m,i}\left(t,s-\frac 1n\right)
=
\lambda^m\left(i+\frac1n\right)a^i t^{in}.
\end{align*}
It follows that $f_{m,i}(t,s)$ is linear in $s$, and its coefficient of
$s$ is
$
\lambda^{m}(1+in)a^i t^{in}.
$
Then we can write
\begin{align}\label{eq3.242}
f_{m,i}(t,s)=\lambda^ma^i (1+in)t^{in}s+B_{m,i}(t)
\end{align}
for some $B_{m,i}(t)\in\mathbb C[t]$ and $B_{0,0}(t)=0$. For simplicity, we denote
$B_{i}(t)=B_{0,i}(t)$.

For $m\neq0$, using
$
[L_{0,i},L_{m,0}]1=m\left(i+\frac1n\right)L_{m,i}1
$
  and Lemma~\ref{lem3.2}, we obtain
\begin{align}\label{eq:fmi-from-f0i}
f_{m,i}(t,s)
=\lambda^m\left(a^it^{in}
\left((1+in)s+\frac1n h_m(nt)-mi\right)+\frac{1}{1+in}tB_i^\prime(t)\right).
\end{align}
Comparing \eqref{eq3.242} and \eqref{eq:fmi-from-f0i}, we immediately get
\begin{align}\label{eq3.262}
    B_{m,i}(t)=\lambda^m\left(a^it^{in}
\left(\frac1n h_m(nt)-mi\right)+\frac{1}{1+in}tB_i^\prime(t)\right).
\end{align}
for $m\neq0$.
Substituting  \eqref{eq3.242}  into
$[L_{1,i},L_{-1,j}]1=-(i+j+\frac{2}{n})L_{0,i+j}1$, we deduce
\begin{align*}
&(n(i+j)+2)\left(a^{i+j}(1+(i+j)n)t^{(i+j)n}s+B_{i+j}(t)\right)
\\=&2a^{i+j}(1+in)(1+jn)t^{(i+j)n}s+\lambda^{-1}a^j(1+jn)t^{jn}B_{1,i}(t)
\\&+\lambda a^i(1+in)t^{in}B_{-1,j}(t)+\lambda a^it^{1+in}B_{-1,j}^\prime(t)
+\lambda^{-1} a^jt^{1+jn}B_{1,i}^\prime(t)
\\&
+a^{i+j}(1+jn)jnt^{(i+j)n}(s-\frac{1}{n})+a^{i+j}(1+in)int^{(i+j)n}(s+\frac{1}{n}).
\end{align*}
Comparing the constant terms with respect to $s$, we have
\begin{align*}
(n(i+j)+2)B_{i+j}(t)
 =&\lambda^{-1}a^j(1+jn)t^{jn}B_{1,i}(t)
+\lambda a^i(1+in)t^{in}B_{-1,j}(t)
\\&+\lambda a^it^{1+in}B_{-1,j}^\prime(t)
+\lambda^{-1} a^jt^{1+jn}B_{1,i}^\prime(t)
\\&
-\frac{1}{n}a^{i+j}(1+jn)jnt^{(i+j)n} +\frac{1}{n}a^{i+j}(1+in)int^{(i+j)n}.
\end{align*}
Inserting \eqref{eq3.262} into the above equation, we verify that
\begin{align}\label{Bij3.27}
&(ni+1)(nj+1)(n(i+j)+2)B_{i+j}(t)\\
={}&a^i(in+1)t^{in+1}\left((in+2)B^\prime_j(t)+tB^{\prime\prime}_j(t)\right)\nonumber\\
&+a^j(jn+1)t^{jn+1}\left((jn+2)B^\prime_i(t)+tB^{\prime\prime}_i(t)\right).\nonumber
\end{align}
Letting $j=0$ in  the above equation, one has
$
t^2B_i''(t)+2tB_i'(t)=(1+in)(2+in)B_i(t).
$

The polynomial solutions of this differential equation are
\begin{align}\label{Bi328}
   B_i(t)=b_it^{1+in}
\end{align}
for $b_i\in\mathbb{C}$ and $b_0=0$. Taking \eqref{Bi328} into \eqref{Bij3.27}, we have $a^ib_j+a^jb_i=b_{i+j}$. Setting $b=b_1$ gives
 $b_i=i b a^{i-1}$ for $i\in\mathbb Z_{+}$.
It follows that
$
f_{0,i}(t,s)=a^i t^{in}(1+in)s+i b a^{i-1}t^{1+in}.
$
Substituting \eqref{Bi328} into \eqref{eq:fmi-from-f0i}, we see that
\begin{align*}
f_{m,i}(t,s)
=
\lambda^m a^i t^{in}
\left((1+in)s+\frac1n h_m(nt)-mi\right)
+i b\lambda^m a^{i-1}t^{1+in}
\end{align*} 
for all $(m,i)\in\mathbb{Z}\times\mathbb{N}$.
\end{proof}

\begin{lemma}\label{lem:q-minus-one-L-action}
Suppose $q=-1$. Then
$
f_{m,1}(t,s)=\lambda^m\bigl(\beta_2+m\Theta_{\mathbf h,\beta_1}(t)\bigr),
f_{0,2}(t,s)=\gamma_2
$
and
$
f_{m,i}(t,s)=0
$
for all remaining $(m,i)\in\mathbb Z\times\mathbb Z_+$ with $i\ge2$.
\end{lemma}

\begin{proof}
By Lemma~\ref{lem:q-minus-one-W-vanishing}, we know that
$
g_{m,1}(t)=\lambda^m\beta_1,
g_{0,2}(t)=\gamma_1$  and $
g_{m,2}(t)=0 \ (m\neq0)$.
For nonzero $m$, applying $[L_{0,1},L_{m,0}]1=0$ and Lemma~\ref{lem3.2},
  we have
  \begin{align}\label{eq3.25}
      \big(s-h_m(-t)\big)\big(f_{0,1}(t,s+m)-f_{0,1}(t,s)\big)=-m(t+m\alpha)\frac{\partial}{\partial t}f_{0,1}(t,s+m).
  \end{align}
Write
  $
  f_{0,1}(t,s)=\sum_{k=0}^pA_k(t)s^k
  $
  with $A_p(t)\neq0$.
  Comparing the coefficient of $s^p$ in \eqref{eq3.25}, we obtain
$ pmA_p(t)=-m(t+m\alpha)A_p'(t)$
for all $m\in\mathbb{Z}$. This gives
  $
  (t+m\alpha)A_p'(t)+pA_p(t)=0,
 $
  which has no nonzero polynomial solution for $p>0$. Therefore,
$
f_{0,1}(t,s)=f_{0,1}(t).
$
Using \eqref{eq3.25} again gives
$
f^\prime_{0,1}(t)=0,
$
namely,
$f_{0,1}(t)=\beta_2\in\mathbb C.$

For any  $m,n\in\mathbb{Z}$, according to
$[L_{m,1},W_{n,0}]1=mW_{m+n,1}1$,
we deduce that
$(t+n\alpha)\big(f_{m,1}(t,s)-f_{m,1}(t,s+n)\big)=0,$ which forces $f_{m,1}(t,s)=f_{m,1}(t)$.
Set $\Theta=\Theta_{\mathbf h,\beta_1}$, and a prime denoting differentiation with respect to $t$. By   \eqref{eq:hm-h1-relation}, we check
\begin{align}\label{eq:hm-theta-derivative}
\beta_1h^\prime_{-m}(-t)
 =&-m\beta_1h^\prime_1(-t)+m(-m-1)\alpha\beta_1\Big(\frac{h_1(-t)-h_1(\alpha)}{-t-\alpha}\Big)^\prime
 \\=&\nonumber -m(m+1)\alpha\Theta^\prime(t)+m\Theta(t)+(t+\alpha)m\Theta^\prime(t)
\\=&m\bigl((t-m\alpha)\Theta'(t)+\Theta(t)\bigr).\nonumber
\end{align}
For $m\neq0$, from
$[L_{-m,0},L_{m,1}]1=-mL_{0,1}1$, we obtain
$
\lambda^{m}\beta_2
=f_{m,1}(t)+(t-m\alpha)f^\prime_{m,1}(t)-\lambda^{m}\beta_1 h^\prime_{-m}(-t).$
From  \eqref{eq:hm-theta-derivative}, we obtain
\begin{equation}\label{eq:fm1-differential}
f_{m,1}(t)+(t-m\alpha)f'_{m,1}(t)
=\lambda^m\bigl(\beta_2+m\Theta(t)+m(t-m\alpha)\Theta'(t)\bigr).
\end{equation}
Let
$
F(t)=f_{m,1}(t)-\lambda^m\bigl(\beta_2+m\Theta(t)\bigr).
$
It follows from  \eqref{eq:fm1-differential} that
$F(t)+(t-m\alpha)F'(t)=0$. Its only polynomial solution is zero. Therefore
\begin{align*}
f_{m,1}(t,s)
    =\lambda^m\bigl(\beta_2+m\Theta(t)\bigr),\quad \forall m\in\mathbb{Z}.
\end{align*}
It remains to consider $i\ge2$. For $m+n\neq0$, the relation $[L_{m,2},W_{n,0}]1=0$ gives $f_{m,2}(t,s)=f_{m,2}(t)$ for all $m\in\mathbb{Z}$. From
$[L_{m,2},L_{-m,0}]1=0$,
we obtain
$
f_{m,2}(t,s)=0$
for  $m\neq0$.
Moreover, since $[L_{m,0},L_{0,2}]1=0$ for
$m\neq0$, we obtain
$
m\lambda^m(t+m\alpha)f'_{0,2}(t)=0.
$
Thus $f_{0,2}(t)=\gamma_2\in\mathbb C$.
Finally, applying
$
[L_{m,i},L_{1,1}]1
=
(i-1)L_{m+1,i+1}1,
$
we conclude that
$f_{m,i}(t,s)=0$
for all $(m,i)$ with $i\ge3$. This proves the lemma.
\end{proof}

\subsubsection{Main theorem}
In this subsection, we shall show the main results of the present  paper.
 \begin{theorem}\label{thm3.4}
Let $M\in\operatorname{Ob}(\mathcal F_1(\mathcal B(q)))$. Then $M$ is isomorphic to one of the following
modules:
\begin{align*}
M\cong
\begin{cases}
\Omega(\lambda,\alpha,\mathbf h(\alpha)),
& \text{if } -2q\notin\mathbb N,\\
\Omega_n(\lambda,a,b,\mathbf h(0)),
& \text{if } q=\dfrac1n \text{ for some } n\in\mathbb Z_{+},\\
\Omega(\lambda,\alpha,\gamma,\mathbf h(\alpha)),
& \text{if } -2q\in\mathbb N \text{ and } q\neq -1,\\
\Omega(\lambda,\alpha,\beta,\gamma,\mathbf h(\alpha)),
& \text{if } q=-1,
\end{cases}
\end{align*}
where $\lambda\in\mathbb C^*$, $\alpha,a,b\in\mathbb C$,
$\beta=(\beta_1,\beta_2)\in\mathbb C^2$,
$\gamma=(\gamma_1,\gamma_2)\in\mathbb C^2$, $\mathbf h(\alpha)\in\mathcal T_\alpha$.
\end{theorem}
\begin{proof}
Restricting $M$ to the subalgebra $\mathcal W$, \eqref{action}
gives the actions of $L_{m,0}$, $W_{m,0}$, $C_1$, and $C_2$.

If $q\neq -1$ and $q\neq \frac1n$, Lemmas
\ref{lem:ordinary-W-vanishing} and \ref{lem:ordinary-L-vanishing}
show that only the possible scalar actions $W_{0,-2q}$ and
$L_{0,-2q}$ remain. Hence $M\cong \Omega(\lambda,\alpha,\mathbf h(\alpha))$
if $-2q\notin\mathbb N$, and
$M\cong\Omega(\lambda,\alpha,\gamma,\mathbf h(\alpha))$ if $-2q\in\mathbb N$.

If $q=\frac1n$, then Lemma~\ref{lem:positive-reciprocal-W-action}
shows that for $\alpha\neq0$ the higher $W$-actions vanish, so
Lemma~\ref{lem:ordinary-L-vanishing} applies. For $\alpha=0$,
Lemmas~\ref{lem:positive-reciprocal-W-action} and
\ref{lem:positive-reciprocal-L-action} give
$M\cong\Omega_n(\lambda,a,b,\mathbf h(0))$.

If $q=-1$, Lemmas~\ref{lem:q-minus-one-W-vanishing} and
\ref{lem:q-minus-one-L-action} give exactly the actions defining
$\Omega\big(\lambda,\alpha,\beta,\gamma,\mathbf h(\alpha)\big)$.

This completes the proof.
\end{proof}
\section{Tensor product modules of $\mathcal{B}(q)$}
Since $t\mathbb C[t,s]$ is a proper submodule of
$\Omega_n(\lambda,a,b,\mathbf h(0))$, this module is not simple.
We therefore consider only tensor products of the simple rank-one
free modules obtained in Theorem~\ref{thm3.4}.

We first recall the definition of  restricted modules (also called smooth modules).
\begin{definition}
Let $\mathcal{L}=\bigoplus_{m\in\mathbb{Z},\,i\in\mathbb{N}}\mathcal{L}_{m,i}$ be an arbitrary $\mathbb{Z}$‑graded Lie algebra.
An $\mathcal{L}$‑module $V$ is called a restricted module if for any $v\in V$ there exists $n\in\mathbb{N}$ such that $\mathcal{L}_{m,i}v=0$ for $m>n$ and any $i\in\mathbb{N}$.
\end{definition}
For a simple restricted $\mathcal{B}(q)$-module $V$, we denote the tensor product
\begin{align*}
T=
\begin{cases}
\bigotimes_{i=1}^{m}\Omega(\lambda_{i},\alpha_{i},\mathbf{h}_{i}(\alpha_i))\otimes V, & -2q\notin \mathbb{N},\\
\bigotimes_{i=1}^{m}\Omega(\lambda_{i},\alpha_{i},\gamma_{i},\mathbf{h}_{i}(\alpha_i))\otimes V, & -2q \in \mathbb{N} \,\text{and} \,q\neq -1,\\
\bigotimes_{i=1}^{m}\Omega(\lambda_{i},\alpha_{i},\beta_{i},\gamma_{i},\mathbf{h}_{i}(\alpha_i))\otimes V, & q=-1,
\end{cases}
\end{align*}
where
$\beta_i=(\beta_{1,i},\beta_{2,i})$ and
$\gamma_i=(\gamma_{1,i},\gamma_{2,i})$. Fix this notation in this section. Now we consider the irreducibility of the tensor product module $T$.

We define a partial order on $\mathbb{Z}^{n}$:
\begin{align*}
(x_{1},x_{2},\ldots,x_{n})<(y_{1},y_{2},\ldots,y_{n})
&\iff \text{there exists }1\leq r\leq n\text{ such that}\\
&\qquad x_{i}=y_{i}\text{ for }i<r\text{ and }x_{r}<y_{r}.
\end{align*}
Each nonzero element $f\in T$ can be uniquely written in the form
\begin{align*}
    f=\sum_{(\mathbf r,\mathbf p)\in S}\mathbf t^{\mathbf r}\mathbf s^{\mathbf p}\otimes v_{\mathbf r,\mathbf p},
\end{align*}
where $S$ is a finite subset of $\mathbb{N}^{2m}$ and $v_{\mathbf r,\mathbf p}$ are nonzero elements in $V$. Define the degree of $f$ in $\mathbf s$ by $\deg_{\mathbf s}(f)=\mathbf p=(p_{1},\ldots,p_{m})$ and its degree in $\mathbf t$ by $\deg_{\mathbf t}(f)=\mathbf r=(r_{1},\ldots,r_{m})$, where $\mathbf s^{\mathbf p}$ and $\mathbf t^{\mathbf r}$ are maximal with respect to this order among the monomials appearing in the sum.

We use the following technical result.
\begin{proposition}\label{prop4.1}
   Let $P$ be a vector space over $\mathbb{C}$ and $P_{1}$ a subspace of $P$. Suppose that $\lambda_{1},\lambda_{2},...,\lambda_{s}\in\mathbb{C}^{*}$ are pairwise distinct, $v_{ij}\in P$ and $f_{i,j}(t)\in \mathbb{C}[t]$ with $\deg f_{ij}(t)=j$ for $i=1,2,...,s;\,j=0,1,2,...,k$. If
   \begin{align*}
       \sum_{i=1}^{s}\sum_{j=0}^{k}\lambda_{i}^{m}f_{ij}(m)v_{ij}\in P_{1} \;\text{for all }m\in\mathbb{Z}\text{ with }m>K.
   \end{align*}
   Then $v_{ij}\in P_{1}$ for all $i,j$.
\end{proposition}

\subsection{Simplicity}

\begin{theorem}\label{thm4.2}
    Suppose that $\lambda_{1},...,\lambda_{m}$ are pairwise distinct and $V$ is irreducible. Then $1\otimes v$ generates the module $T$ for any nonzero element $v\in V$.
\end{theorem}
\begin{proof}
    Fix any nonzero $v\in V$. Let $X$ be the submodule of $T$ generated by $1\otimes v$. Suppose $t^{\mathbf{r}}s^{\mathbf{p}}\otimes v\in X$. Since $V$ is a restricted module, there exists a positive integer $K$ such that $W_{k,0}v=L_{k,0}v=0$ for all $k>K$. Then we have
    \begin{align*}
    W_{k,0}(\mathbf t^{\mathbf r}\mathbf s^{\mathbf p}\otimes v)
=\sum_{i=1}^m\Big(\lambda_i^k(t_i-kq\alpha_i)t_i^{r_i}(s_i-kq)^{p_i}
\prod_{\ell\ne i}t_\ell^{r_\ell}s_\ell^{p_\ell}\Big)\otimes v.\\
    \end{align*}
By Proposition \ref{prop4.1}, we obtain $t_i\mathbf t^{\mathbf r}\mathbf s^{\mathbf p}\otimes v\in X$ for all $1\leq i\leq m$. Similarly, by
\begin{align*}
        L_{k,0}(\mathbf t^{\mathbf r}\mathbf s^{\mathbf p}\otimes v)
=&\sum_{i=1}^m\Big(\Big(\lambda_i^k
\bigl(s_i+qh_{i,k}(q^{-1}t_i)\bigr)t_i^{r_i}(s_i-kq)^{p_i}\\
&-\lambda_i^kkqr_i
(t_i-kq\alpha_i)t_i^{r_i-1}(s_i-kq)^{p_i}\Big)\cdot\prod_{\ell\ne i}t_\ell^{r_\ell}s_\ell^{p_\ell}\Big)\otimes v,
    \end{align*}
    $s_i\mathbf t^{\mathbf r}\mathbf s^{\mathbf p}\otimes v\in X$ for all $i$. Therefore $\mathbb{C}[\mathbf{t},\mathbf{s}]\otimes v\subseteq X$.

    Now denote $V'=\{w\in V\mid\mathbb{C}[\mathbf{t},\mathbf{s}]\otimes w\subseteq X\}$. The previous argument shows $V'\neq 0$. For any $w\in V'$, $x\in \mathcal{B}(q)$, we have
    \begin{align*}
        x(f(\mathbf{t},\mathbf{s})\otimes w)=xf(\mathbf{t},\mathbf{s})\otimes w+f(\mathbf{t},\mathbf{s})\otimes xw\in X,
    \end{align*}
which implies $f(\mathbf{t},\mathbf{s})\otimes xw\in X$, and $xw\in V'$. Therefore $V'$ is a submodule of $V$. Since $V$ is irreducible, $V'=V$, and hence $X=T$.
\end{proof}
\begin{theorem}\label{thm4.3}
    Suppose that $\lambda_{1},\ldots,\lambda_{m}$ are pairwise distinct, $V$ is an irreducible restricted module, for each $i$, either $\alpha_i\neq0$, or $q=-1$ and
$\beta_{1,i}\neq0$. Then $T$ is an irreducible $\mathcal{B}(q)$-module.
\end{theorem}
\begin{proof}
  We will divide the proof into two parts.

  {\bf Case 1}.  $q\neq -1$.

  Let $X$ be a nonzero submodule of $T$, and choose a nonzero $f\in X$ of minimal total degree in the variables $s_1,\ldots,s_m$. For all sufficiently large $k$, the $i$-th exponential component of $W_{k,0}f$ is
\begin{align*}
(t_i-kq\alpha_i)
f(t_1,s_1,\ldots,t_i,s_i-kq,\ldots,t_m,s_m).
\end{align*}
Since $\alpha_i\neq0$, the highest relevant coefficient in $k$ lowers the $s_i$-degree. Proposition~\ref{prop4.1} contradicts the minimality of $f$. Thus $f=f(t_1,\ldots,t_m)$. Choose such an $f$ of minimal total degree in $t_1,\ldots,t_m$.

For each $1\leq i\leq m$, comparing the coefficient of
$k^2\lambda_i^k$ in $L_{k,0}f$ yields
\begin{align}\label{eq:t-derivative-coefficient}
\alpha_i\left(
-q\frac{h_{i,1}(q^{-1}t_i)-h_{i,1}(\alpha_i)}
{q^{-1}t_i-\alpha_i}f
+q^2\frac{\partial f}{\partial t_i}
\right)\in X.
\end{align}
On the other hand, the constant term in the $i$-th exponential component of
$W_{k,0}f$ implies that $t_if\in X$. Then we get $p(t_i)f\in X$ for all
$p(t_i)\in\mathbb C[t_i]$. In particular,
\[
\frac{h_{i,1}(q^{-1}t_i)-h_{i,1}(\alpha_i)}
{q^{-1}t_i-\alpha_i}f\in X.
\]
Since $\alpha_iq\neq0$, subtracting this term in \eqref{eq:t-derivative-coefficient}, we obtain
$\frac{\partial f}{\partial t_i}\in X$. 
The minimality of the degree
of $f$ in $t_1,\ldots,t_m$ shows that $f=1\otimes v$ for some
$0\neq v\in V$. By Theorem~\ref{thm4.2}, $X=T$.

{\bf Case 2}. $q= -1$.

Let $X$ be a nonzero submodule of $T$ and choose a nonzero $f\in X$ of minimal degree in $\mathbf s$. For every $i$ with $\alpha_i\neq0$, by the argument in Case 1, we get the results. Now we only consider $\alpha_i=0$ and  $\beta_{1,i}\neq0$. Then for all sufficiently large $k$,
\begin{align*}
W_{k,1}f=\sum_{i=1}^m\lambda_i^k\beta_{1,i}
f(t_1,s_1,\ldots,t_i,s_i+k,\ldots,t_m,s_m).
\end{align*}
Proposition~\ref{prop4.1} gives the same contradiction whenever the degree in $s_i$ is positive. Hence $f=f(t_1,\ldots,t_m)$.

Choose such an $f$ of minimal degree in $\mathbf t$. Since
\begin{align*}
W_{k,0}f=\sum_{i=1}^m\lambda_i^k(t_i+k\alpha_i)f,
\end{align*}
we have $t_if\in X$ for every $i$. Therefore $p(t_1,\ldots,t_m)f\in X$ for every polynomial $p$. Moreover,
\begin{align*}
L_{k,1}f=\sum_{i=1}^m\lambda_i^k\left(
\bigl(\beta_{2,i}+k\Theta_{\mathbf h_i,\beta_{1,i}}(t_i)\bigr)f
+k\beta_{1,i}\frac{\partial f}{\partial t_i}\right).
\end{align*}
If $\beta_{1,i}\neq0$, Proposition~\ref{prop4.1} implies
$\Theta_{\mathbf h_i,\beta_{1,i}}(t_i)f+\beta_{1,i}\frac{\partial f}{\partial t_i}\in X$, then $\frac{\partial f}{\partial t_i}\in X$.   The minimality of the degree in $t$ yields $f=1\otimes v$ for some nonzero $v\in V$. By Theorem~\ref{thm4.2}, $X=T$.
\end{proof}

\begin{corollary}
Let $m\in\mathbb N$ and $\lambda_i\in\mathbb C^*$ be pairwise distinct. Assume that either $\alpha_i\neq0$, or $q=-1$ and
$\beta_{1,i}\neq0$ for each $i$. Then $T$ is an irreducible $\mathcal B(q)$-module.
\end{corollary}
\begin{proof}
    This follows from Theorem \ref{thm4.3} by taking $V$ to be the one-dimensional trivial module.
\end{proof}
\begin{theorem}
    Let $N_{n}$ be the subspace of $\Omega(\lambda,\alpha_{1},\beta_1,\gamma_1,\mathbf{h})\otimes \Omega(\lambda,\alpha_{2},\beta_2,\gamma_2,\mathbf{g})$ spanned by the elements $s_{1}^{r}(s_{1}+s_{2})^{l}f_{1}(t_1)f_{2}(t_2)$, where $0\leq r \leq n$, $l\in\mathbb{N}$ and $f_1,f_2\in\mathbb C[t]$. Then $N_{n}$ is a proper submodule of $\Omega(\lambda,\alpha_{1},\beta_1,\gamma_1,\mathbf{h})\otimes \Omega(\lambda,\alpha_{2},\beta_2,\gamma_2,\mathbf{g})$ for any $n\in\mathbb{Z}_{+}$.
\end{theorem}
\begin{proof}
Set $x=s_1$ and $y=s_1+s_2$, then 
$N_n=\{F\in\mathbb C[t_1,t_2,x,y]\mid \deg_xF\leq n\}.$
For $0\leq r\leq n$, the action of $W_{k,0}$ on
$x^ry^lf_1(t_1)f_2(t_2)$ is
\begin{align*}
\lambda^k(y-kq)^l\left(
(t_1-kq\alpha_1)(x-kq)^r
+(t_2-kq\alpha_2)x^r\right)f_1(t_1)f_2(t_2),
\end{align*}
which has $x$-degree at most $r$. In the $L_{k,0}$-action, the only terms that could have $x$-degree $r+1$ combine as
\begin{align}\label{eq:L_k0_action}
\lambda^k(y-kq)^l
\left(x(x-kq)^r+(y-x)x^r\right)f_1(t_1)f_2(t_2).
\end{align}
Since $x((x-kq)^r-x^r)$ has $x$-degree at most $r$, $x$-degree in \eqref{eq:L_k0_action} is at most $r$. It follows that $L_{k,0}N_n,W_{k,0}N_n\subseteq N_n$. When $q=-1$, the $i=1$ actions preserve $N_n$ because they only shift $s_1,s_2$ and apply polynomial differential operators in $t_1,t_2$. All remaining actions vanish. So $N_n$ is a submodule. And it is proper since $x^{n+1}\notin N_n$.

This completes the proof.
\end{proof}
If one rank one factor is not simple, the tensor product of a proper submodule of that factor with all remaining factors and $V$ is proper. The same argument applies if $V$ is reducible. Combining with the previous results, we conclude that
\begin{theorem}\label{thm4.6}
    The module $T$ is irreducible if and only if $\lambda_1,\ldots,\lambda_m$ are pairwise distinct, $V$ is an irreducible restricted $\mathcal B(q)$-module, and, for each $1\leq i\leq m$, either $\alpha_i\neq0$, or $q=-1$ and $\beta_{1,i}\neq0$.
\end{theorem}
\subsection{Isomorphism classes}
We determine when two such tensor product modules are isomorphic.

For brevity, denote by
\begin{align*}
T'=
\begin{cases}
\bigotimes_{j=1}^{n}\Omega(\mu_j,b_j,\mathbf g_j)\otimes V',
&-2q\notin\mathbb N,\\
\bigotimes_{j=1}^{n}\Omega(\mu_j,b_j,\gamma'_j,\mathbf g_j)\otimes V',
&-2q\in\mathbb N\ \text{and}\ q\neq-1,\\
\bigotimes_{j=1}^{n}
\Omega(\mu_j,b_j,\beta'_j,\gamma'_j,\mathbf g_j)\otimes V',
&q=-1.
\end{cases}
\end{align*}
Write
\begin{align*}
\mathbf h_i=\{h_{i,k}(t)\mid k\in\mathbb Z\}\in\mathcal T_{\alpha_i},
\qquad
\mathbf g_j=\{g_{j,k}(t)\mid k\in\mathbb Z\}\in\mathcal T_{b_j}.
\end{align*}
If $-2q\in\mathbb N$, set
\begin{align*}
\Gamma(T)=\sum_{i=1}^m\gamma_i,\qquad
\Gamma(T')=\sum_{j=1}^n\gamma'_j,
\end{align*}
where $\gamma_i=(\gamma_{1,i},\gamma_{2,i})$ and
$\gamma'_j=(\gamma'_{1,j},\gamma'_{2,j})$.

When $-2q\in\mathbb N$, for
$\delta=(\delta_1,\delta_2)\in\mathbb C^2$, denote by $V^\delta$
the module with the same underlying vector space as $V$, on which
\begin{align*}
W_{0,-2q}\mathbin{\cdot_\delta}v=W_{0,-2q}v+\delta_1v,\qquad
L_{0,-2q}\mathbin{\cdot_\delta}v=L_{0,-2q}v+\delta_2v,
\end{align*}
while all other operators act as on $V$.

\begin{theorem}\label{thm:4.8}
Let $T$ and $T'$ be irreducible tensor product modules as above.
Then $T\cong T'$ if and only if, after a permutation of the tensor
factors of $T'$, the following conditions hold:
\begin{enumerate}
\item $m=n$;
\item $\lambda_i=\mu_i,\ \alpha_i=b_i,\ \mathbf h_i=\mathbf g_i$
for all $1\leq i\leq m$;
\item if $q=-1$, then $\beta_i=\beta'_i$ for all $1\leq i\leq m$;
\item if $-2q\in\mathbb N$, then
$
V'\cong V^{\Gamma(T)-\Gamma(T')}
$
as $\mathcal B(q)$-modules;

if
$-2q\notin\mathbb N$, then
$
V'\cong V
$
as $\mathcal B(q)$-modules.
\end{enumerate}
\end{theorem}

\begin{proof}
\setcounter{claim}{0}
Suppose that $\phi:T\to T'$ is a $\mathcal B(q)$-module
isomorphism.

\begin{claim}
We have $m=n$, and after reordering the tensor factors of $T'$,
$\lambda_i=\mu_i$ for $1\leq i\leq m$.
\end{claim}
Write
\begin{align*}
\sigma_i^kf(\mathbf t,\mathbf s)
&=f(t_1,s_1,\ldots,t_i,s_i-kq,\ldots,t_m,s_m),\\
\sigma_j'^kg(\mathbf t',\mathbf s')
&=g(t_1',s_1',\ldots,t_j',s_j'-kq,\ldots,t_n',s_n').
\end{align*}
Since $V$ and $V'$ are restricted, for every $f\in T$ and all
sufficiently large $k$, we have
\begin{align*}
W_{k,0}f=\sum_{i=1}^m
\lambda_i^k(t_i-kq\alpha_i)\sigma_i^kf.
\end{align*}
The corresponding formula holds on $T'$. Hence
\begin{align*}
\sum_{i=1}^m\lambda_i^k
\phi\bigl((t_i-kq\alpha_i)\sigma_i^kf\bigr)
=\sum_{j=1}^n\mu_j^k
(t_j'-kqb_j)\sigma_j'^k\phi(f).
\end{align*}
For fixed $f$, the coefficients of the exponential terms are
polynomials in $k$ with values in $T'$. If $f\neq0$, the
coefficients on both sides are nonzero. After grouping equal
exponential bases, Proposition~\ref{prop4.1} shows that every
$\lambda_i$ occurs among the $\mu_j$'s and vice versa. Hence
\begin{align*}
\{\lambda_1,\ldots,\lambda_m\}
=\{\mu_1,\ldots,\mu_n\}.
\end{align*}
Thus $m=n$, and after reordering the tensor factors of $T'$,
\begin{align*}
\lambda_i=\mu_i,\qquad 1\leq i\leq m.
\end{align*}

\begin{claim}
For every $f\in T$ and $1\leq i\leq m$,
\begin{align*}
\phi(t_if)=t_i'\phi(f),\qquad
\phi(s_if)=s_i'\phi(f).
\end{align*}
\end{claim}
For every $f\in T$ and all sufficiently large $k$, we have
\begin{align*}
\sum_{i=1}^m\lambda_i^k
\phi\bigl((t_i-kq\alpha_i)\sigma_i^kf\bigr)=\phi(W_{k,0}f)=
W_{k,0}\phi(f)
=\sum_{i=1}^m\lambda_i^k
(t_i'-kqb_i)\sigma_i'^k\phi(f).
\end{align*}
By Proposition~\ref{prop4.1}, for each $1\leq i\leq m$, we have
\begin{equation}\label{eq:iso-W-component}
\phi\bigl((t_i-kq\alpha_i)\sigma_i^kf\bigr)
=(t_i'-kqb_i)\sigma_i'^k\phi(f).
\end{equation}
Both sides are polynomials in $k$. Comparing their constant terms, we get
\begin{equation}\label{eq:iso-t-linear}
\phi(t_if)=t_i'\phi(f).
\end{equation}

Similarly, for all sufficiently large $k$,
\begin{align*}
\phi(L_{k,0}f)
&=\sum_{i=1}^m\lambda_i^k\phi\left(
\left(s_i+qh_{i,k}(q^{-1}t_i)
-kq(t_i-kq\alpha_i)\frac{\partial}{\partial t_i}\right)
\sigma_i^kf\right),\\
L_{k,0}\phi(f)
&=\sum_{i=1}^m\lambda_i^k
\left(s_i'+qg_{i,k}(q^{-1}t_i')
-kq(t_i'-kqb_i)\frac{\partial}{\partial t_i'}\right)
\sigma_i'^k\phi(f).
\end{align*}
Since $\phi(L_{k,0}f)=L_{k,0}\phi(f)$,
Proposition~\ref{prop4.1} shows that the corresponding summands are
equal. Both sides are polynomials in $k$. Since
$h_{i,0}=g_{i,0}=0$, comparison of their constant terms gives
\begin{equation}\label{eq:iso-s-linear}
\phi(s_if)=s_i'\phi(f).
\end{equation}
Then equations~\eqref{eq:iso-t-linear}--\eqref{eq:iso-s-linear} show that for $p\in\mathbb C[\mathbf t,\mathbf s]$,
\begin{align*}
\phi(p(\mathbf t,\mathbf s)f)
=p(\mathbf t',\mathbf s')\phi(f).
\end{align*}

\begin{claim}
For every $1\leq i\leq m$,
$\alpha_i=b_i$ and $\mathbf h_i=\mathbf g_i$. Moreover,
$
\phi(1\otimes v)=1\otimes\tau(v)
$
for a vector space isomorphism $\tau: V \xrightarrow{\sim} V'$.
\end{claim}
Fix $1\leq i\leq m$. For $v\in V$, set
\begin{align*}
F_v(\mathbf t',\mathbf s')=\phi(1\otimes v).
\end{align*}
Applying \eqref{eq:iso-W-component} to $1\otimes v$, we obtain
\begin{equation}\label{eq:iso-F-shift}
	(t_i'-kq\alpha_i)F_v(\mathbf t',\mathbf s')
	=(t_i'-kqb_i)F_v(\mathbf t',\mathbf s'-kq\mathbf e_i).
\end{equation}
Setting $t_i'=0$ and $k\gg 0$, we get
\begin{align*}
\alpha_i F_v(\mathbf t',\mathbf s')\big|_{t_i'=0}
=b_iF_v(\mathbf t',\mathbf s'-kq\mathbf e_i)\big|_{t_i'=0}.
\end{align*}
If
$F_v(\mathbf t',\mathbf s')\big|_{t_i'=0}\neq0,$
then by comparing the highest coefficients in $s_i'$ we get
$\alpha_i=b_i$. Then if $\alpha_i\neq b_i$, we get
$F_v(\mathbf t',\mathbf s')\big|_{t_i'=0}=0$
for every $v\in V$. That is to say $F_v\in  t_i'\mathbb{C}[\mathbf t',\mathbf s']$, which implies
$\operatorname{Im}\phi\subseteq t_i'T'$,
contrary to the surjectivity of $\phi$. Therefore
$\alpha_i=b_i$.

Now \eqref{eq:iso-F-shift} reduces to
\begin{align*}
F_v(\mathbf t',\mathbf s')
=F_v(\mathbf t',\mathbf s'-kq\mathbf e_i).
\end{align*}
Since $kq\neq0$, $F_v$ is independent of $s_i'$. Repeating this
argument for every $i$, we conclude that $F_v=F_v(\mathbf t')$.

For all sufficiently large $k$,  we have
\begin{align*}
\phi\bigl(L_{k,0}(1\otimes v)\bigr)
=\sum_{r=1}^m\lambda_r^k
\bigl(s_r'+qh_{r,k}(q^{-1}t_r')\bigr)F_v,
\end{align*}
\begin{align*}
L_{k,0}\phi(1\otimes v)
=\sum_{r=1}^m\lambda_r^k
\left(\bigl(s_r'+qg_{r,k}(q^{-1}t_r')\bigr)F_v
-kq(t_r'-kq\alpha_r)\frac{\partial F_v}{\partial t_r'}\right).
\end{align*}
For each $1\leq i\leq m$, Proposition~\ref{prop4.1} gives, for all
sufficiently large integers $k$,
\begin{equation}\label{eq:iso-h-differential}
\bigl(g_{i,k}(q^{-1}t_i')-h_{i,k}(q^{-1}t_i')\bigr)F_v
=k(t_i'-kq\alpha_i)\frac{\partial F_v}{\partial t_i'}.
\end{equation}
The difference of the two sides is a polynomial in $k$ which vanishes
for infinitely many integers. It is therefore the zero polynomial, so
\eqref{eq:iso-h-differential} is valid for any $k\in\mathbb{Z}$. Fix $i$. Write
\begin{align*}
F_v=\sum_{\ell=1}^N p_\ell(t_i')u_\ell,
\end{align*}
where $p_\ell(t_i')\in\mathbb C[t_i']\setminus\{0\}$ and $u_1,\ldots,u_N$ are $\mathbb C$-linearly independent
elements of
\begin{align*}
\mathbb C[t_1',\ldots,\widehat{t_i'},\ldots,t_m']\otimes V'.
\end{align*}
Taking $k=1$ in \eqref{eq:iso-h-differential} and
comparing the coefficients of $u_\ell$, we get
\begin{align*}
\frac{p_\ell'(t_i')}{p_\ell(t_i')}(t_i'-q\alpha_i)
=g_{i,1}(q^{-1}t_i')-h_{i,1}(q^{-1}t_i').
\end{align*}
If $p_\ell$ had a root different from $q\alpha_i$, then
$(t_i'-q\alpha_i)p_\ell'/p_\ell$ would have a pole at that root, whereas
$g_{i,1}(q^{-1}t_i')-h_{i,1}(q^{-1}t_i')$ is a polynomial. It forces for some $r\in\mathbb Z_+$,
\begin{align*}
p_\ell(t_i')=c_\ell(t_i'-q\alpha_i)^r,
\end{align*}
\begin{align*}
g_{i,1}(q^{-1}t_i')-h_{i,1}(q^{-1}t_i')=r.
\end{align*}
Since $g_{i,1}(q^{-1}t_i')-h_{i,1}(q^{-1}t_i')$ is independent of
$v$ and $\ell$, the same $r$ occurs for every nonzero coefficient
of every $F_v$. According to \eqref{eq:iso-h-differential}, if $r>0$,
$\operatorname{Im}\phi\subseteq(t_i'-q\alpha_i)^rT'$, contrary to the surjectivity
of $\phi$. It follows that $r=0$ and
$
\frac{\partial F_v}{\partial t_i'}=0,h_{i,1}=g_{i,1}
$.

By \eqref{eq:hm-h1-relation}, $\mathbf h_i=\mathbf g_i$. Repeating
this argument for every $i$, for a vector space isomorphism $\tau:V\to V'$, we obtain
$\phi(1\otimes v)=1\otimes\tau(v)$.
\begin{claim}
If $q=-1$, then $\beta_i=\beta'_i$ for $1\leq i\leq m$.
\end{claim}
Assume that $q=-1$. For all sufficiently large $k$,
\begin{align*}
\phi\bigl(W_{k,1}(1\otimes v)\bigr)
=\sum_{r=1}^m\lambda_r^k\beta_{1,r}F_v,\qquad
W_{k,1}\phi(1\otimes v)
=\sum_{r=1}^m\lambda_r^k\beta'_{1,r}F_v.
\end{align*}
Proposition~\ref{prop4.1} gives
$\beta_{1,i}F_v=\beta'_{1,i}F_v$ for each $1\leq i\leq m$, hence $\beta_{1,i}=\beta'_{1,i}$. Since
$\mathbf h_i=\mathbf g_i$, we also have
$\Theta_{\mathbf h_i,\beta_{1,i}}
=\Theta_{\mathbf g_i,\beta'_{1,i}}$.

For all sufficiently large $k$ we obtain
\begin{align*}
\phi\bigl(L_{k,1}(1\otimes v)\bigr)
&=\sum_{r=1}^m\lambda_r^k
\bigl(\beta_{2,r}+k\Theta_{\mathbf h_r,\beta_{1,r}}(t_r')\bigr)F_v,\\
L_{k,1}\phi(1\otimes v)
&=\sum_{r=1}^m\lambda_r^k
\bigl(\beta'_{2,r}+k\Theta_{\mathbf g_r,\beta'_{1,r}}(t_r')\bigr)F_v.
\end{align*}
By Proposition~\ref{prop4.1}, for each $1\leq i\leq m$,
\begin{align*}
\bigl(\beta_{2,i}+k\Theta_{\mathbf h_i,\beta_{1,i}}(t_i')\bigr)F_v
=\bigl(\beta'_{2,i}+k\Theta_{\mathbf g_i,\beta'_{1,i}}(t_i')\bigr)F_v.
\end{align*}
Therefore $\beta_{2,i}=\beta'_{2,i}$, thus
$\beta_i=\beta'_i$.

\begin{claim}
If $-2q\notin\mathbb N$, $V'\cong V$ as $\mathcal B(q)$-modules; if
$-2q\in\mathbb N$,
$V'\cong V^{\Gamma(T)-\Gamma(T')}$ as
$\mathcal B(q)$-modules.
\end{claim}
For $x\in[\mathcal B(q),\mathcal B(q)]$, Claims 1--4 give
\begin{align*}
0=\phi\bigl(x(1\otimes v)\bigr)-x\phi(1\otimes v)
=1\otimes\bigl(\tau(xv)-x\tau(v)\bigr).
\end{align*}
Then $\tau(xv)=x\tau(v)$ for all
$x\in[\mathcal B(q),\mathcal B(q)]$ and $v\in V$.
If $-2q\notin\mathbb N$, $[\mathcal B(q),\mathcal B(q)]=\mathcal B(q)$, then we have $V'\cong V$. Now let
$-2q\in\mathbb N$.
Applying $W_{0,-2q}$ and $L_{0,-2q}$ to $\tau(v)$, we can see
\begin{align*}
W_{0,-2q}\tau(v)
=\tau(W_{0,-2q}v)
+\left(\sum_{i=1}^m\gamma_{1,i}
-\sum_{i=1}^m\gamma'_{1,i}\right)\tau(v),
\end{align*}
\begin{align*}
L_{0,-2q}\tau(v)
=\tau(L_{0,-2q}v)
+\left(\sum_{i=1}^m\gamma_{2,i}
-\sum_{i=1}^m\gamma'_{2,i}\right)\tau(v).
\end{align*}
It follows that $V'\cong V^{\Gamma(T)-\Gamma(T')}$.

Conversely, after reordering the tensor factors of $T'$, the result
follows directly.
\end{proof}

\section{Applications to subalgebras of $\mathcal B(-1)$}
In this section, we consider applications of the preceding results to subalgebras of $\overline{\mathcal B}(-1)$. By Theorem~\ref{thm3.4}, $C_1$ and $C_2$ act trivially on every module in $\mathcal F_1(\mathcal B(-1))$. Therefore if $\mathfrak a$ is a Lie subalgebra of $\overline{\mathcal B}(-1)$, the quotient map and the inclusion $\mathfrak a\hookrightarrow\overline{\mathcal B}(-1)$ induce the restriction functor
\begin{align*}
\operatorname{Res}^{\mathcal B(-1)}_{\mathfrak a}:
\mathcal B(-1)\text{-}\mathrm{Mod}\longrightarrow
\mathfrak a\text{-}\mathrm{Mod},
\qquad M\longmapsto M|_{\mathfrak a}.
\end{align*}
\subsection{The Heisenberg--Virasoro subalgebra $\mathcal{H}$ and its Takiff Lie algebra}

The centerless Heisenberg--Virasoro algebra $\mathcal{H}$ is the Lie algebra
with basis
$
\{L_m,I_m\mid m\in\mathbb Z\}
$
subject to the following Lie brackets:
\begin{equation*}
\begin{aligned}
[L_m,L_n]=(n-m)L_{m+n},\
[L_m,I_n]=nI_{m+n},\
[I_m,I_n]=0.
\end{aligned}
\end{equation*}
It is isomorphic to the subalgebra $\mathcal{s}=\operatorname{span}\{-L_{m,0},\,L_{m,1}\mid m\in\mathbb Z\}$ of $\overline{\mathcal B}(-1)$ under the identification
\begin{align*}
L_m\mapsto -L_{m,0},\qquad I_m\mapsto L_{m,1}.
\end{align*}
The Takiff Lie algebra of $\mathcal{H}$ is defined by
$
\widetilde{\mathcal{H}}:=\mathcal{H}\otimes \mathbb C[x]/(x^2),$ 
with Lie bracket
\begin{align*}
[a\otimes x^i,\ b\otimes x^j]=[a,b]\otimes x^{i+j},
\qquad a,b\in\mathcal H,\ i,j\in\{0,1\}.
\end{align*}

It is isomorphic to the subalgebra $\overline{\mathcal{s}}=\operatorname{span}\{-L_{m,0},\,L_{m,1},\,-W_{m,0},\,W_{m,1}\mid m\in\mathbb Z\}$
of $\overline{\mathcal B}(-1)$ under the identification
\begin{align*}
L_m\mapsto -L_{m,0},\qquad I_m\mapsto L_{m,1},\qquad
\overline{L}_m\mapsto -W_{m,0},\qquad \overline{I}_m\mapsto W_{m,1},
\end{align*}
which is exactly the BMS-Kac-Moody algebras.
The above embeddings give restriction functors
\begin{align*}
\mathcal F_1(\mathcal B(-1))
\xrightarrow{\operatorname{Res}^{\mathcal B(-1)}_{\widetilde{\mathcal H}}}
\widetilde{\mathcal H}\text{-}\mathrm{Mod}
\xrightarrow{\operatorname{Res}^{\widetilde{\mathcal H}}_{\mathcal H}}
\mathcal H\text{-}\mathrm{Mod}.
\end{align*}
Their composite satisfies
\begin{align*}
\operatorname{Res}^{\mathcal B(-1)}_{\mathcal H}
=
\operatorname{Res}^{\widetilde{\mathcal H}}_{\mathcal H}
\circ
\operatorname{Res}^{\mathcal B(-1)}_{\widetilde{\mathcal H}}.
\end{align*}
Restriction functor does not preserve irreducibility in general, so the irreducibility of the resulting modules must be considered separately.

\begin{definition}\label{def5.1}
For $\lambda\in\mathbb C^*$ and $\alpha,\beta_1,\beta_2\in\mathbb C$, $\mathbf{h}=\{h_{n}(t)\mid n\in\mathbb{Z}\}\in\mathcal{T}_{\alpha}$, let
$\Psi(\lambda,\alpha,\beta_1,\beta_2,\mathbf h)=\mathbb C[t,s]$.

\begin{enumerate}
    \item The $\mathcal{H}$-module structure on $\Psi(\lambda,\alpha,\beta_1,\beta_2,\mathbf h)$ is defined by
    \begin{align*}
	         L_{m} f(t,s)=&-\lambda^{m}(s- h_{m}(-t))f(t,s+m)-m\lambda^{m}(t+m \alpha)\frac{\partial}{\partial t}f(t,s+m),\\
	     I_{m} f(t,s)=& \lambda^{m}\bigl(\beta_{2}+m\Theta_{\mathbf h,\beta_1}(t)\bigr)f(t,s+m)
	     +m\lambda^{m}\beta_{1}\frac{\partial}{\partial t}f(t,s+m).
	    \end{align*}

    \item The $\widetilde{\mathcal{H}}$-module structure on $\Psi(\lambda,\alpha,\beta_1,\beta_2,\mathbf h)$ is defined by the above $\mathcal{H}$-action together with
     \begin{align*}
	        \overline{L}_m f(t,s)=-\lambda^{m}(t+m\alpha)f(t,s+m),\ 
	        \overline{I}_m f(t,s)=\lambda^{m}\beta_{1}f(t,s+m)
	    \end{align*}
\end{enumerate}
\end{definition}
Applying restriction functors to the modules defined by \eqref{C3.1}--\eqref{W3.3}, we get
\begin{align*}
\operatorname{Res}^{\mathcal B(-1)}_{\widetilde{\mathcal H}}
\Omega\big(\lambda,\alpha,\beta,\gamma,\mathbf h(\alpha)\big)
\cong
\Psi(\lambda,\alpha,\beta_1,\beta_2,\mathbf h(\alpha)).
\end{align*}
From Theorems \ref{thm3.5}, \ref{prop3.2}, \ref{thm4.6} and \ref{thm:4.8}, it is easy to see that the $\widetilde{\mathcal{H}}$-module $\Psi(\lambda,\alpha,\beta_1,\beta_2,\mathbf h)$ has the following results, which also given in \cite{CG26} and \cite{CG27}.
\begin{theorem}\label{thm5.2}
For the $\widetilde{\mathcal{H}}$-module $\Psi(\lambda,\alpha,\beta_1,\beta_2,\mathbf h)$, we have the following:
\begin{enumerate}
    \item $\Psi(\lambda_{1},\alpha_{1},\beta_{1},\beta_{2},\mathbf{h})\cong\Psi(\lambda_{2},\alpha_{2},\beta_{1}',\beta_{2}',\mathbf{g})$ if and only if
   $ \lambda_{1}=\lambda_{2},\,
    \alpha_{1}=\alpha_{2},\,
    \beta_{1}=\beta_{1}',\,
    \beta_{2}=\beta_{2}'$ and $
    \mathbf{h}=\mathbf{g}$.
    \item $\Psi(\lambda,\alpha,\beta_{1},\beta_{2},\mathbf{h})$ irreducible if and only if
    $
    \alpha\neq 0 \, \text{or} \, \beta_{1}\neq 0.
    $
    \item $\bigotimes_{i=1}^{m}\Psi(\lambda_{i},\alpha_{i},\beta_{1,i},\beta_{2,i},\mathbf{h}_{i})\otimes V$
    is irreducible if and only if $\lambda_{1},\ldots,\lambda_{m}$ are pairwise distinct,
    $\alpha_{i}\neq 0$ or $\beta_{1,i}\neq 0$ for each $i$, and $V$ is an irreducible
    restricted $\widetilde{\mathcal H}$-module.
\end{enumerate}
\end{theorem}

\subsection{Admissible tensor product modules of $\mathcal{H}$}
Although $\Psi(\lambda,\alpha,\beta_1,\beta_2,\mathbf h)$ is naturally obtained by restriction, its irreducibility as an $\mathcal{H}$-module cannot be deduced directly from the previous results. We therefore restrict to the admissible case and construct a family of admissible tensor product modules.
\begin{definition}
    An $\mathcal{H}$-module $M$ is called \textbf{admissible} provided that
\begin{align*}
I_{m+n}v = I_{m}I_{n}v
\end{align*}
for all $m,n \in \mathbb{Z}$ and $v \in M$.
\end{definition}
A direct computation shows the following result.
\begin{proposition}
    $\Psi(\lambda,\alpha,\beta_{1},\beta_{2},\mathbf{h})$ is an admissible $\mathcal{H}$-module if and only if $\beta_{1}=0$ and $\beta_{2}\in\{0,1\}$.
\end{proposition}
When $\beta_2=0$, all $I_m$ act trivially. In the following, we
consider the nontrivial admissible case $\beta_1=0$ and $\beta_2=1$.

Define two subalgebras of $\mathcal{H}$ by $Q=\operatorname{span}\{L_{m},I_{n}\mid m\geq -1,n\geq 0\}$ and $P=\operatorname{span}\{L_{m},I_{n}\mid m\geq 0,n\geq 0\}$. Let $V$ be a $Q$-module such that for any $v\in V$, $L_{m}v=0, I_{n}v=0$ for all but finitely many of $m\geq -1$ and $n\geq0$.

 Define an $\mathcal{H}$-action on the vector space $T(V,\Psi)=V\otimes \Psi(\lambda,\alpha,0,1,\mathbf{h})$ as follows. For $m\in \mathbb{Z}$, $v\in V,f(t,s)\in\mathbb{C}[t,s]$,
\begin{align}
    L_{m}(v\otimes f(t,s))=&v\otimes L_{m} f(t,s)+(e^{mt}-1)\frac{d}{dt}v\otimes I_{m} f(t,s),\label{5.1}\\
    I_{m}(v\otimes f(t,s))=&e^{mt}v\otimes I_{m}f(t,s),\label{5.2}
\end{align}
where $e^{mt}=\sum_{i=0}^{\infty}\frac{m^{i}}{i!}I_{i}$, $(e^{mt}-1)\frac{d}{dt}=\sum_{i=0}^{\infty}\frac{m^{i}}{i!}L_{i-1}-L_{-1}$. Set $g(m)=(e^{mt}-1)\frac{d}{dt}$, $h(m)=e^{mt}$, then it follows that $[g(m),h(n)]=nh(m+n)-nh(n)$ and $[g(m),g(n)]=mg(m)-ng(n)+(n-m)g(m+n)$.
\begin{proposition}
    $T(V,\Psi)$ is an $\mathcal{H}$-module under the actions \eqref{5.1} and \eqref{5.2}.
\end{proposition}
\begin{proof}
 For any $m,n\in\mathbb{Z}$ and $f\in\mathbb{C}[t,s]$, we have 
\begin{align*}
&(L_{m}L_{n}-L_{n}L_{m})(v\otimes f)\\
=&L_{m}(v\otimes L_{n}f+(e^{nt}-1)\tfrac{d}{dt}v\otimes I_{n}f)-L_{n}(v\otimes L_{m}f+(e^{mt}-1)\tfrac{d}{dt}v\otimes I_{m}f)\\
=&v\otimes L_{m}L_{n}f+(e^{mt}-1)\tfrac{d}{dt}v\otimes I_{m}L_{n}f+(e^{nt}-1)\tfrac{d}{dt}v\otimes L_{m}I_{n}f\\
&+(e^{mt}-1)\tfrac{d}{dt}(e^{nt}-1)\tfrac{d}{dt}v\otimes I_{m}I_{n}f-v\otimes L_{n}L_{m}f-(e^{nt}-1)\tfrac{d}{dt}v\otimes I_{n}L_{m}f\\
&-(e^{mt}-1)\tfrac{d}{dt}v\otimes L_{n}I_{m}f-(e^{nt}-1)\tfrac{d}{dt}(e^{mt}-1)\tfrac{d}{dt}v\otimes I_{n}I_{m}f\\
=&v\otimes(n-m)L_{m+n}f-m(e^{mt}-1)\tfrac{d}{dt}v\otimes I_{m+n}f+n(e^{nt}-1)\tfrac{d}{dt}v\otimes I_{m+n}f\\
&+(e^{mt}-1)\tfrac{d}{dt}(e^{nt}-1)\tfrac{d}{dt}v\otimes I_{m+n}f-(e^{nt}-1)\tfrac{d}{dt}(e^{mt}-1)\tfrac{d}{dt}v\otimes I_{m+n}f
\\=&(n-m)L_{m+n}(v\otimes f).
\end{align*}
Similarly, we get 
\begin{align*}
    &(L_{m}I_{n}-I_{n}L_{m})(v\otimes f)\\
    =&L_{m}(e^{nt}v\otimes I_{n}f)-I_{n}(v\otimes L_{m}f+(e^{mt}-1)\tfrac{d}{dt}v\otimes I_{m}f)\\
    =&e^{nt}v\otimes L_{m}I_{n}f+(e^{mt}-1)\tfrac{d}{dt}e^{nt}v\otimes I_{m}I_{n}f
    \\&-e^{nt}v\otimes I_{n}L_{m}f-e^{nt}(e^{mt}-1)\tfrac{d}{dt}v\otimes I_{n}I_{m}f\\
    =&ne^{(m+n)t}v\otimes I_{m}I_{n}f=nI_{m+n}(v\otimes f)
\end{align*}
and 
\begin{align*}
    &(I_{m}I_{n}-I_{n}I_{m})(v\otimes f)\\
    =&I_{m}(e^{nt}v\otimes I_{n}f)-I_{n}(e^{mt}v\otimes I_{m}f)\\
    =&e^{mt}e^{nt}v\otimes I_{m}I_{n}f-e^{nt}e^{mt}v\otimes I_{n}I_{m}f\\
    =&[e^{mt},e^{nt}]v\otimes I_{m+n}f=0.
\end{align*}
The proposition holds.
\end{proof}
\begin{lemma}\label{lem:admissible-shift}
For any $i\in \mathbb Z_+$, we have 
$
e^{mt}L_{-1}^i=(L_{-1}-m)^i e^{mt}.
$
Consequently,
\begin{align*}
\Bigl((e^{mt}-1)\frac{d}{dt}\Bigr)L_{-1}^i
=
(L_{-1}-m)^i e^{mt}\frac{d}{dt}-L_{-1}^{i+1}.
\end{align*}
\end{lemma}
\begin{proof}
We prove the first identity by induction on $i$. Since
$
I_nL_{-1}=L_{-1}I_n-nI_{n-1},
$
we have
\begin{align*}
e^{mt}L_{-1}
=
\sum_{n=0}^\infty \frac{m^n}{n!}I_nL_{-1}
=
L_{-1}e^{mt}-\sum_{n=1}^\infty \frac{m^n}{(n-1)!}I_{n-1}
=
(L_{-1}-m)e^{mt}.
\end{align*}
Thus the assertion holds for $i=1$. Assume that
$
e^{mt}L_{-1}^i=(L_{-1}-m)^i e^{mt}
$,
then
\begin{align*}
e^{mt}L_{-1}^{i+1}
&=e^{mt}L_{-1}^iL_{-1}
=(L_{-1}-m)^i e^{mt}L_{-1}
=(L_{-1}-m)^{i+1}e^{mt}.
\end{align*}
Therefore, for all $i\in \mathbb Z_+$,
$
e^{mt}L_{-1}^i=(L_{-1}-m)^i e^{mt}.
$
Finally,
we obtain
\begin{align*}
\Bigl((e^{mt}-1)\frac{d}{dt}\Bigr)L_{-1}^i
=
e^{mt}\frac{d}{dt}L_{-1}^i-L_{-1}^{i+1}
=
(L_{-1}-m)^i e^{mt}\frac{d}{dt}-L_{-1}^{i+1}.
\end{align*}
\end{proof}

Let $\mathcal{C}$ denote the category of all non-trivial $Q$-modules satisfying the following condition:

For any $0\neq v\in V$, there exists $r\in
\mathbb{Z}_{+}$ such that $L_{r+i}v=I_{r+i}v=0$ for all $i\ge 1.$

The minimal such $r$ is called the order of $v$, denoted by ${\rm ord}(v)$. Choose $s\in \mathbb{Z}_{+}$ minimal such that the set $\{v\in V | L_{s+i}v=I_{s+i}v=0 \, \text{for all} \,i\ge 1\}$ is non-zero. Denote this set by $V_{\mathfrak b}$. Then $V_{\mathfrak b}$ is a $P$-module.
\begin{proposition}\label{prop4.6}
Let $V$ be an irreducible $Q$-module in $\mathcal{C}$.
 Then the following statements hold:

\begin{enumerate}
    \item For each $p\in\mathbb Z_+$, the map
    $ L_{-1}^p\big|_{V_{\mathfrak{b}}}:V_{\mathfrak{b}}\to V$ is injective.
    \item $V=\bigoplus_{i\geq 0}L_{-1}^iV_{\mathfrak{b}}$.
    In particular, $ V\cong \mathbb C[L_{-1}]\otimes V_{\mathfrak{b}}$ as vector spaces.
\end{enumerate}
\end{proposition}
\begin{proof}
{\rm (1)}
Let $v\in V_{\mathfrak{b}}$ satisfy $L_{-1}^pv=0.$
    Applying $L_{s+p}$ and $I_{s+p}$ to both sides, we obtain
    \begin{align*}
    L_{s+p}L_{-1}^pv=0,\qquad I_{s+p}L_{-1}^pv=0.
    \end{align*}
Meanwhile,
    \begin{align*}
   & L_{s+p}L_{-1}^pv
    =
    (-1)^p\prod_{j=2}^{p+1}(s+j)L_sv, \
 I_{s+p}L_{-1}^pv
    =
    (-1)^p\prod_{j=1}^{p}(s+j)I_sv.
    \end{align*}
    It follows that
    $L_sv=I_sv=0$.
Thus $v$ is annihilated by $L_j$ and $I_j$ for all $j\geq s$. The minimality of $s$ gives $v=0$.

{\rm (2)}
Set $U:=\sum_{i\geq 0}L_{-1}^iV_{\mathfrak{b}}$.
We claim that $U$ is a $Q$-submodule of $V$.

Clearly, $L_{-1}U\subseteq U$. For $m\geq 0$, by repeating use of
$
[L_m,L_{-1}]=-(m+1)L_{m-1},
$
we have that $L_mL_{-1}^iv$ is a linear combination of terms $L_{-1}^jL_rv$, where
$0\leq j\leq i$ and $r\geq 0$. Since $V_{\mathfrak{b}}$ is stable under $P$, it follows that $L_mU\subseteq U$.

Similarly, by
$
I_nL_{-1}=L_{-1}I_n-nI_{n-1},
$
we see that $I_nL_{-1}^iv$ is a linear combination of terms $L_{-1}^jI_rv$, where
$0\leq j\leq i$ and $r\geq 0$, thus $I_nU\subseteq U$.

Hence $U$ is a nonzero $Q$-submodule of $V$. Since $V$ is irreducible, we conclude that $U=V$. Therefore,
\begin{align*}
V=\sum_{i\geq 0}L_{-1}^iV_{\mathfrak{b}}.
\end{align*}
It remains to show that the above sum is direct. We just need to prove for every $p\geq 1$,
\begin{align*}
L_{-1}^{p}V_{\mathfrak{b}}\cap \sum_{i=0}^{p-1}L_{-1}^iV_{\mathfrak{b}}=0.
\end{align*}
Now let $0\neq L_{-1}^{p}v\in L_{-1}^{p}V_{\mathfrak{b}}$, where $v\in V_{\mathfrak{b}}$. Then
$
L_{s+p}L_{-1}^{p}v\neq 0
$ or
$
I_{s+p}L_{-1}^{\,p}v
\neq 0.
$
However, for every $0\leq i\leq p-1$,
$
L_{s+p}L_{-1}^{i}V_{\mathfrak{b}}=0$
and
$I_{s+p}L_{-1}^{i}V_{\mathfrak{b}}=0$, a contradiction.
Thus
\begin{align*}
L_{-1}^{p}V_{\mathfrak{b}}\cap \sum_{i=0}^{p-1}L_{-1}^{i}V_{\mathfrak{b}}=0.
\end{align*}
Then
$
\sum_{i\geq 0}L_{-1}^{i}V_{\mathfrak{b}}
=
\bigoplus_{i\geq 0}L_{-1}^{i}V_{\mathfrak{b}},
$
and
$
V=\bigoplus_{i\geq 0}L_{-1}^{i}V_{\mathfrak{b}}.
$
\end{proof}

\begin{theorem}
    Let $\lambda\in \mathbb{C}^{*},\alpha\in \mathbb{C}$, $\mathbf h\in\mathcal T_\alpha$ and let $V$ be an irreducible module in $\mathcal C$. Then $T(V,\Psi)$ has a series of $\mathcal{H}$-submodules
     \begin{align*}
        V_{\mathfrak b}^{(0)}\subsetneq V_{\mathfrak b}^{(1)}\subsetneq \ldots \subsetneq V_{\mathfrak b}^{(p)}\subsetneq \ldots,
    \end{align*}
    where $V_{\mathfrak b}^{(p)}=\sum_{i=0}^{p}L_{-1}^{i}V_{\mathfrak b}\otimes \mathbb{C}[t,s]$.
\end{theorem}
\begin{proof}
    We prove by induction on $p$. It is easy to check $V_{\mathfrak b}^{(0)}$ is a submodule of $T(V,\Psi)$. Suppose that $V_{\mathfrak b}^{(0)},\ldots,V_{\mathfrak b}^{(p-1)}$ are all $\mathcal{H}$-submodules. Take $L_{-1}^{p}v\otimes f+v^{(p-1)}\in V_{\mathfrak b}^{(p)}$, where $v\in V_{\mathfrak b},v^{(p-1)}\in V_{\mathfrak b}^{(p-1)}$. By Lemma \ref{lem:admissible-shift} and \eqref{5.1}, we obtain
    \begin{align*}
        L_{m}&(L_{-1}^{p}v\otimes f+v^{(p-1)})\\
        =&L_{-1}^{p}v\otimes L_{m}f+(e^{mt}-1)\frac{d}{dt} L^{p}_{-1} v\otimes I_{m}f+L_{m}(v^{(p-1)})\\
        =&L_{-1}^{p}v\otimes L_{m}f+((L_{-1}-m)^{p}e^{mt}\frac{d}{dt}-L_{-1}^{p+1})v\otimes I_{m}f+L_{m}(v^{(p-1)})\\
        =&L_{-1}^{p}v\otimes L_{m}f+((L_{-1}-m)^{p}(L_{-1}+\sum_{i=1}^{\infty}\frac{m^{i}}{i!}L_{i-1})-L_{-1}^{p+1}))v\otimes I_{m}f\\
        &+L_{m}(v^{(p-1)})\in V_{\mathfrak b}^{(p)},\\
        I_{m}&(L_{-1}^{p}v\otimes f+v^{(p-1)})\\
        =&e^{mt}L_{-1}^{p}v\otimes I_{m}f+I_{m}(v^{(p-1)})\\
        =&(L_{-1}-m)^{p} e^{mt}v\otimes I_{m}f+I_{m}(v^{(p-1)})\\
        =&(L_{-1}-m)^{p}\sum_{i=0}^{\infty}\frac{m^{i}}{i!}I_{i}v\otimes I_{m}f+I_{m}(v^{(p-1)})\in V_{\mathfrak b}^{(p)}.
    \end{align*}
This shows that $V_{\mathfrak b}^{(p)}$ is a submodule of $T(V,\Psi)$.
\end{proof}

\begin{corollary}
For each $p\geq 0$, with $V_{\mathfrak b}^{(-1)}=0$, we have 
\begin{align*}
V_{\mathfrak{b}}^{(p)}/V_{\mathfrak b}^{(p-1)}
\cong
V_{\mathfrak{b}}\otimes \mathbb C[t,s]
\end{align*}
as vector spaces.
\end{corollary}
\begin{proof}
Define
$
\varphi_p:V_{\mathfrak{b}}\otimes \mathbb C[t,s]\longrightarrow V_{\mathfrak{b}}^{(p)}/V_{\mathfrak{b}}^{(p-1)}
$ by
$
\varphi_p(v\otimes f(t,s))
=
L_{-1}^pv\otimes f(t,s)+V_{\mathfrak{b}}^{(p-1)}$,
for $v\in V_{\mathfrak{b}},\ f(t,s)\in\mathbb C[t,s]$.
It is clearly surjective. Moreover, by $(1)$ of Proposition \ref{prop4.6} and
$L_{-1}^pV_{\mathfrak{b}}\cap \sum_{i=0}^{p-1}L_{-1}^iV_{\mathfrak{b}}=0,
$
we get $\varphi_p$ is injective.
\end{proof}
Since $V_{\mathfrak{b}}^{(p)}/V_{\mathfrak{b}}^{(p-1)}$ is isomorphic to $V_{\mathfrak{b}}^{(0)}$ as vector spaces, it induces a new $\mathcal H$-module structure on $V_{\mathfrak{b}}^{(0)}$ for each $p\geq 0$:
\begin{align}
I_m\ast_p(v\otimes f)=&e^{mt}v\otimes I_mf,\label{eq:Imp}\\
L_m\ast_p(v\otimes f)
=&
v\otimes L_mf+
\left(\left((e^{mt}-1)\frac{d}{dt}-pm\right)v\right)\otimes I_mf.\label{eq:Lmp}
\end{align}
Let $V_{\mathfrak b}^{\langle p\rangle}$ be the $P$-module with underlying vector space $V_{\mathfrak b}$, whose action is given by
\begin{align*}
L_0^{\langle p\rangle}=L_0-p\,\mathrm{Id},\qquad
L_m^{\langle p\rangle}=L_m\ (m>0),\qquad
I_n^{\langle p\rangle}=I_n.
\end{align*}
Combining the above vector space isomorphism with the induced $\mathcal{H}$-module structure, we obtain the following theorem.
\begin{theorem}
    For each $p\geq 0$, the quotient module $V_{\mathfrak b}^{(p)}/V_{\mathfrak b}^{(p-1)}$ is isomorphic to $V_{\mathfrak b}^{\langle p\rangle}\otimes\mathbb C[t,s]$ endowed
with the $\mathcal H$-action $\ast_p$ defined by
\eqref{eq:Imp}--\eqref{eq:Lmp}, where $V_{\mathfrak{b}}^{(-1)}=0$.
\end{theorem}
\subsection{The rank one extended $W(2,2)$ algebra $\mathcal T$}
Recall that the rank one extended $W(2,2)$ algebra $\mathcal T$ is the Lie algebra with basis
\begin{align*}
\{L_m,W_m,H_m\mid m\in\mathbb Z\},
\end{align*}
subject to the relations
\begin{equation*}
\begin{aligned}
&[L_m,L_n]=(n-m)L_{m+n},\quad
[L_m,W_n]=(n-m)W_{m+n},\\&
[L_m,H_n]=nH_{m+n},\
[W_m,W_n]=[W_m,H_n]=[H_m,H_n]=0.
\end{aligned}
\end{equation*}
It can be realized as a subalgebra of $\overline{\mathcal B}(-1)$. Indeed, it is isomorphic to the subspace
\begin{align*}
\mathcal{D}=\operatorname{span}\{-L_{m,0},\,W_{m,0},\,W_{m,1}\mid m\in\mathbb Z\}.
\end{align*}
The inclusion $\mathcal T\cong\mathcal D\hookrightarrow \overline{\mathcal B}(-1)$ induces the restriction functor
\begin{align*}
\operatorname{Res}^{\mathcal B(-1)}_{\mathcal T}:
\mathcal F_1(\mathcal B(-1))\longrightarrow
\mathcal T\text{-}\mathrm{Mod}.
\end{align*}

Therefore, by restricting the $\mathcal {B}(-1)$-module
$\Omega\big(\lambda,\alpha,\beta,\gamma,\mathbf h(\alpha)\big)$ defined by \eqref{C3.1}--\eqref{W3.3}
to the subalgebra $\mathcal T$, we obtain a $\mathcal T$-module, which we define as follows.
 \begin{definition}\label{def:rank-one-extended-W-module}
     For $\lambda\in \mathbb{C}^{*}$, $\alpha,\beta_{1}\in\mathbb{C}$ and $\mathbf h\in\mathcal T_\alpha$, define  the $\mathcal T$-module structure on $\Psi(\lambda,\alpha,\beta_{1},\mathbf{h})=\mathbb{C}[t,s]$ as follows
     \begin{align*}
         L_{m} f(t,s)=&-\lambda^{m}(s- h_{m}(-t))f(t,s+m)-m\lambda^{m}(t+m \alpha)\frac{\partial}{\partial t}f(t,s+m),\\
     W_{m} f(t,s)=& \lambda^{m}(t+m\alpha)f(t,s+m),\
     H_{m} f(t,s)=\lambda^{m}\beta_{1}f(t,s+m).
	     \end{align*}
 \end{definition}
 Applying the restriction functor to
$\Omega\big(\lambda,\alpha,\beta,\gamma,\mathbf h(\alpha)\big)$ under the actions \eqref{C3.1}--\eqref{W3.3} gives
\begin{align*}
\operatorname{Res}^{\mathcal B(-1)}_{\mathcal T}
\Omega\big(\lambda,\alpha,\beta,\gamma,\mathbf h(\alpha)\big)
\cong
\Psi(\lambda,\alpha,\beta_1,\mathbf h).
\end{align*}
 \begin{theorem}
 The following holds:
 \begin{enumerate}
     \item $\Psi(\lambda_{1},\alpha_{1},\beta_{1},\mathbf{h})\cong\Psi(\lambda_{2},\alpha_{2},\beta_{1}',\mathbf{g})$ if and only if
   $\lambda_{1}=\lambda_{2},\,
    \alpha_{1}=\alpha_{2},\,
    \beta_{1}=\beta_{1}'$ and $
    \mathbf{h}=\mathbf{g}.$
    \item $\Psi(\lambda,\alpha,\beta_{1},\mathbf{h})$ is an irreducible $\mathcal T$-module if and only if $\alpha\neq 0$.
    \item $\bigotimes_{i=1}^{m}\Psi(\lambda_{i},\alpha_{i},\beta_{1,i},\mathbf{h}_{i})\otimes V$ is irreducible if and only if $\lambda_{1},...,\lambda_{m}$ are pairwise distinct, $\alpha_{i}\neq 0$ for each $i$ and $V$ is an irreducible restricted $\mathcal T$-module.
 \end{enumerate}
 \end{theorem}
\begin{proof}
The first statement follows immediately from Proposition \ref{prop3.2}.

We prove the second statement. If $\alpha=0$, according to the proof of Theorem \ref{thm3.4}, $t^{i}\mathbb C[t,s]$ is a
proper $\mathcal T$-submodule for each $i\in \mathbb{Z}_{+}$.

Now assume that $\alpha\neq 0$ and let $M$ be a nonzero $\mathcal T$-submodule of
$\mathbb C[t,s]$. We shall prove that $M=\mathbb C[t,s]$.

Take a nonzero element $f(t,s)\in M$, and write
$
f(t,s)=\sum_{i=0}^r a_i(t)s^i,
\, a_r(t)\neq 0.$
For every $m\in\mathbb Z$, we obtain
$
\lambda^{-m}W_m f(t,s)=(t+m\alpha)f(t,s+m)\in M.
$
Expanding the right-hand side as a polynomial in $m$, we obtain
\begin{align*}
\sum_{j=0}^{r+1} m^j f_j(t,s)\in M
\qquad \text{for all }m\in\mathbb Z,
\end{align*}
where $f_j(t,s)\in \mathbb C[t,s]$. Since $M$ is a vector subspace, by Proposition \ref{prop4.1},
we have $f_j(t,s)\in M$ for $0\le j\le r+1$. In particular,
the coefficient of $m^{r+1}$,  $\alpha a_r(t)\in M$.
Therefore $M$ contains a nonzero polynomial $g(t)\in \mathbb C[t]$.

Applying $W_0$ and $L_0$, we obtain $g(t)\mathbb C[t,s]\subseteq M$.
Moreover, as in the proof of Theorem~\ref{thm5.2},
we get $g'(t)\in M$.
It follows that $M=\mathbb C[t,s]$, so $\Psi(\lambda,\alpha,\beta_1,\mathbf h)$ is irreducible.
The third statement follows from Theorems~\ref{thm4.3} and \ref{thm4.6}, together with the first statement.
\end{proof}
This gives another family of irreducible modules over the rank one extended $W(2,2)$ algebra.

\section*{Acknowledgements}
The first author is partially supported by the Fundamental Research Funds for the Central Universities (No.~B250201222) and by the Young Faculty Overseas
Training Program (No.~202506710216), which is jointly funded by the China Scholarship Council and Hohai University.
Part of this paper was written during the first author's visit to Uppsala University, whose hospitality is gratefully acknowledged.
\section*{Authors' contributions}
All authors contributed equally to this work.

\section*{Data Availability Statement}
This manuscript has no associated data.

\section*{Conflicts of Interest}
The authors declare that they have no conflict of interest.

\end{document}